\documentclass[final]{siamart}

\usepackage{lipsum}
\usepackage{amsfonts}
\usepackage{graphicx}
\usepackage{epstopdf}
\usepackage{algorithmic}
\usepackage{amsmath,amsfonts,mathrsfs,amssymb}
\numberwithin{equation}{section}
\numberwithin{theorem}{section}

\def\R{\mathbb{R}}

\newtheorem{remark}{Remark}[section]

\newcommand{\be}{\begin{equation}}
\newcommand{\ee}{\end{equation}}
\newcommand{\ba}{\begin{array}}
\newcommand{\ea}{\end{array}}

\newcommand{\grad}{{\rm \; grad\;}}

\newcommand{\bea}{\begin{eqnarray*}}
\newcommand{\eea}{\end{eqnarray*}}
\newcommand{\bean}{\begin{eqnarray}}
\newcommand{\eean}{\end{eqnarray}}
\def\grad{\nabla}

\def\va{\varepsilon}

\def\cydot{\leavevmode\raise.4ex\hbox{.}}
\newcommand{\normmm}[1]{{\left\vert\kern-0.25ex\left\vert\kern-0.25ex\left\vert #1
    \right\vert\kern-0.25ex\right\vert\kern-0.25ex\right\vert}}

\ifpdf
  \DeclareGraphicsExtensions{.eps,.pdf,.png,.jpg}
\else
  \DeclareGraphicsExtensions{.eps}
\fi

\newcommand{\TheTitle}{Inverse conductivity problem with internal data}
\newcommand{\TheAuthors}{F. Triki, and T. Yin}

\headers{Conductivity}{\TheAuthors}

\title{{\TheTitle}\thanks{This work was supported in part by the grant ANR-17-CE40-0029 of the French National Research Agency ANR (project MultiOnde).}}

\author{
  Faouzi Triki\thanks{ Laboratoire Jean Kuntzmann, UMR CNRS 5224, Universit\'e Grenoble-Alpes, 700 Avenue  Centrale, 38401 Saint-Martin- d'H\`eres, France
    (\email{faouzi.triki@univ-grenoble-alpes.fr}).}
  \and
  Tao Yin\thanks{Department of Computing \& Mathematical Sciences, California Institute of Technology, 1200 East California Blvd., CA 91125, United States (\email{taoyin89@caltech.edu}).}
}

\usepackage{amsopn}

\ifpdf
\hypersetup{
  pdftitle={\TheTitle},
  pdfauthor={\TheAuthors}
}
\fi

\begin{document}

\maketitle

\begin{abstract}
  This paper concerns the reconstruction of a scalar coefficient of a second-order  elliptic equation in divergence form posed on a bounded domain from internal data. This theory finds applications in multi-wave imaging, greedy methods to approximate parameter-dependent elliptic problems, and image treatment with partial differential equations.  We first show that  the inverse problem for smooth coefficients can be rewritten as a linear transport equation. Assuming that the coefficient is known near the boundary, we study the 
  well-posedness of  associated transport  equation as well as its numerical resolution using  discontinuous Galerkin method. We propose a regularized transport equation  that allow us to derive rigorous convergence rates of the numerical method in terms of the order of the polynomial approximation as well as the regularization parameter. We finally provide  numerical examples for the inversion assuming a lower regularity of the coefficient, and using synthetic data.

\end{abstract}



\section{Introduction}
Let $\Omega$ be a $C^6$-smooth bounded domain of $\mathbb{R}^n$, $n=2,3$, with boundary $\Gamma$. Let $\nu(x)$ be the outward normal vector at $x\in \Gamma, $ and \( d =\sup_{x, y\in \Omega} \|x-y\|\) be the diameter of $\Omega$. We set, for $\eta\in (0, d) $, $\sigma_0 \in  W^{2,\infty}(\Omega )$, $\Omega_\eta =\{x \in \Omega, \textrm{dist}(x, \Gamma)> \eta\}, $  and $0<k_1<k_2,$
\bea
\Sigma =\{ \sigma \in W^{2,\infty}(\Omega );\; \sigma= \sigma_0\, \textrm{ in } \Omega\setminus\overline{\Omega_\eta}, \;  k_1\le \sigma,\;
\|\sigma\|_{W^{2,\infty}(\Omega )}\le k_2\}.
\eea
Let $g$ be fixed in $H^{\frac{7}{2}}(\Gamma)$, and satisfy $\int_\Gamma g dx = 0$.
Then, according to the classical elliptic regularity theory
\bean \label{cond}
\mathrm{div}(\sigma \nabla u)=0\; \mathrm{in}\; \Omega,\;\;\;\; \; \sigma\partial_\nu u =g \; \mathrm{on}\; \Gamma,\;\;\int_\Omega u_\sigma dx = 0,
\eean
has a unique solution $u_\sigma\in H^5(\Omega )$,  and there exists a constant $c= c(\Sigma, \Omega)>0$ such that
\bean \label{bound}
\|u_\sigma\|_{H^5(\Omega )}\le c.
\eean

The goal of this work is to study the following  inverse problem (IP):   {\it   Given  $\sigma_0$ and the interior data $u_\sigma|_{\Omega}$, to reconstruct the conductivity $\sigma|_{\Omega}$.}

This inverse problem is of importance in many different scientific and engineering fields including  photoacoustic tomography, studies  of effective properties of composite materials, and approximation of parametric partial differential equations. Photoacoustic tomography is a recent hybrid imaging modality that couples diffusive optical waves with ultrasound waves to achieve high-resolution imaging of optical properties of biological tissues~\cite{AGKNS,CT,BU,BCT}.  The  inverse problem (IP) appears in the second inversion, called quantitative photoacoustic  tomography,  where the derived internal data  is used to recover the optical coefficients of the sample~\cite{Br1,BMT}. Motivated by the search for sharp bounds on the effective moduli of composites many researchers  have considered the problem of characterizing mathematically among all  the gradient fields those solving the equation~\eqref{cond} for some  function $\sigma$ within the set $\Sigma$.  In the  context of approximation of parameter-dependent elliptic problems by greedy algorithms  the  inverse problem (IP) has been considered  with  infinitely many interior data available \cite{CZ}. Hence solving the inverse  problem  with a single datum may reduce the dimensionality of  the set of parameters used to  accurately  approximate  a  targeted  compact set of solutions \cite{PCDDPW}.

Given $\sigma_0$ and the interior data $u_\sigma|_{\Omega}$, the inverse problem can be recasted as a linear steady transport equation satisfied by $\sigma \in \Sigma$,
\bea
\nabla \sigma \cdot \nabla u_\sigma + (\Delta u_\sigma) \sigma = 0 \;\; \mathrm{in}\;\; \Omega.
\eea
The steady  transport equation is one of the basic equations in mathematical physics. It is widely  used in fluid mechanics, for example to model mass transfer~\cite{PS}. From the mathematical  point of view there are several results addressing the well-posdeness of the equation. In order to briefly review some of these results we introduce suitable boundary conditions. To do so we split the  boundary of $\Gamma$ into three disjoint parts, the inflow $\Gamma_{\textrm{in}}$, the outflow set $\Gamma_{\textrm{out}}$, and
 the characteristic set $\Gamma_{0}$, defined by
\bean \label{boundaries} \hspace{0.5cm}
\Gamma_{\textrm{in}} = \{x\in \Gamma: \nabla u_\sigma \cdot \nu<0 \},\,\Gamma_{\textrm{out}} =
 \{x\in \Gamma: \nabla u_\sigma \cdot \nu>0 \},\,\Gamma_0= \Gamma \setminus (\Gamma_{\textrm{in}}  \cup \Gamma_{\textrm{out}}).
\eean
Assuming that $\nabla u_\sigma$ never vanishes in $\Omega$ and using the method of characteristics, one can easily show that the system
\bean \label{transport1}
\nabla \sigma \cdot \nabla u_\sigma + (\Delta u_\sigma) \sigma = 0 \;\; \mathrm{in}\;\; \Omega, \;\; \sigma=\sigma_0 \;\; \mathrm{on}\;\;
\Gamma_{\textrm{in}},
\eean
admits a unique solution. The method of characteristics can not be applied  when the  set of characteristic curves  has a complex structure, for example  when $\grad u_\sigma$ vanishes. In order to overcome this difficulty, many works have considered the case where the lower order term  dominates the transport term. In this framework the theory of linear steady transport equations becomes part of a more general theory of degenerate elliptic equations  (\cite{Fi,KN,OR}, see also Chapter 12 in \cite{PS} and references therein). Let $\kappa >0$  be a fixed constant. When  $n=3$, and  assuming  that the interior data $u_\sigma$ verifies
\bean \label{condomega}
\inf_{x\in \Omega}|\Delta u_\sigma (x)|> \kappa>0,
\eean
it was proved in \cite{OR},  by using  the vanishing viscosity method, that  the system~\eqref{transport1} admits a weak solution  in $L^2(\Omega)$, satisfying
\bea
\| \sigma\|_{L^2(\Omega)} \leq \frac{1}{\kappa}\|\nabla \sigma_0 \cdot \nabla u_\sigma + (\Delta u_\sigma) \sigma_0 \|_{L^2(\Omega)}.
\eea
If, in addition,  $\partial \Gamma_{\textrm{in}}$ is a one-dimensional $C^1$ manifold, then there is a unique weak solution $\sigma \in L^2(\Omega)$. However in general  without  geometrical assumptions on the  characteristic  set, the  system~\eqref{transport1} may have other weak solutions. Indeed the problems of  uniqueness  and  regularity of solutions to the system~\eqref{transport1}  are difficult issues,  most of the available results relate  to the case of multi-connected domains with isolated  inflow boundary $\Gamma_{\textrm{in}}$ \cite{KN,OR}.

Linear steady transport equations is also part of a general theory  developed  by Friedrichs dealing  with symmetric positive systems of first-order linear partial differential equations~\cite{Fr}.  Recently,  numerical  methods  based on discontinuous Galerkin methods  have  renewed interest in Friedrichs systems \cite{EGC,AB,PE12}. In this settings the conditions required  for existence and uniqueness of solutions to the system  are almost similar to the  ones derived in \cite{Fi,KN,OR}.

Due the specificity  of our equation (the lower order term is linked  to the transport speed)  the condition  \eqref{condomega} which is the minimal requirement for the  existence   of solutions in all the presented  variational approaches   is  stronger than
assuming that  $\grad u_\sigma$ does not vanish in $\Omega$. Furthermore, if the interior  data is recovered with zero noise the existence  of solution is  in fact  guaranteed  from  the fact that  the transport  equation is originated  from  the elliptic system \eqref{cond}. The uniqueness of solutions to the  system  \eqref{transport} is  established in \cite{Al1} using elliptic unique continuation properties of the  transport speed $\nabla u_\sigma$.   However, in applications, the internal data $u_\sigma|_{\Omega}$ is usually measured with  limited accuracy. Hence, studying the well-posedness of the inverse problem (IP)  when the measurements are noisy is of critical importance.

We aim in this paper  to study the well-posedness of the inverse problem (IP)  as well as its numerical  resolution. If  the interior measurement is noisy, that is, only  $U^\delta \in H^5(\Omega)$ is available, satisfying
\bean \label{noise}
\|U^\delta- u_\sigma\|_{H^5(\Omega)}  \leq  \delta\|u_\sigma\|_{H^5(\Omega)}, \;\; \int_{\Omega} U^\delta dx=0,
\eean
with $\delta \in (0, 1)$ is small enough, could we construct $\sigma_\delta\in L^2(\Omega)$, solution to the transport equation \eqref{transport1}, with $U^\delta$ substituting $u_\sigma$, that tends   to $\sigma \in \Sigma$ when the noise $\delta $ tends to zero?  This question is also related to the  numerical approximation of the solution  to the inverse problem (IP).

Since in  general  a noisy data  $U^\delta$ does not  belong to
$\mathcal U:=  \big\{u_\sigma \in H^5(\Omega): \sigma \in \Sigma \big\}$,  the transport equation   \eqref{transport} with $u_\sigma$ substituted by $U^\delta$ does not need to have a solution (see example 2.4 in \cite{BMT}). Therefore, it is necessary to find  an adequate regularization method  to  the  linear equation \eqref{transport}. The  vanishing viscosity  method is a well known way to regularize this type of equation, it   consists  in adding a second order term $\va \Delta \cdot$, with $\va$ is a small positive parameter \cite{Al1}. The regularized equation becomes then an elliptic one and  the singularities of targeted function $\sigma$ may not  be  visible in the regularized solution. In this paper we propose a first order regularization method   based on an original regularization of the lower order term that does not  improve the regularity of the solution.
We note that since  the  transport equation \eqref{transport1}  is solved in the weak $L^2$ topology,  the assumption on the regularity of $\sigma$,  $u_\sigma$,   and the noisy data  $U^\delta$ may be significantly reduced.  Indeed, we considered such a large smoothness to ease the presentation  of our theoretical stability estimates. \\

The paper is organized as follows. In section 2, we study the stability of the solution of the inverse problem (IP) with respect to the noise in the internal measurement.  We  prove the existence and uniqueness of solution to the  regularized  transport equation~\eqref{transport5} in Theorem  \ref{existenceuniquenessregularized}. We also show the convergence of the regularized solution to the targeted conductivity distribution in Theorem~\ref{theoremstability}. Section 3 is devoted to the  numerical approximation of the solution using discontinuous Galerkin method.  We derive error estimates of the  numerical approximation in Theorem~\ref{thmerror}. We finally  provide in section 4 several numerical tests  based on synthetic data, and  assuming a lower regularity of the conductivity distribution.

\section{Stability estimates}
In this section we  study  the well-posdeness  of the  regularized transport equation using a variational approach developed in~\cite{Fr}.  We introduce the following modified transport equation
\bean \label{transport}
2\nabla \gamma \cdot \nabla u_\sigma +\gamma \Delta u_\sigma = 0 \; \mathrm{in}\; \Omega
\;\; \mathrm{and}\;\; \gamma = \gamma_0,
\eean
for $\gamma \in H^1(\Omega)$ and $ \gamma_0=\sqrt{\sigma_0}$. We know from the system \eqref{cond} that $\gamma=\sqrt{\sigma}$  is a solution of \eqref{transport}. We deduce from  \cite{Al1,BCT} that $\sqrt{\sigma}$ is indeed the unique solution.

Let $U^\delta \in H^5(\Omega)$ be  given, and satisfying \eqref{noise}.
Our first objective in this section is to derive  for $\delta>0$,
small enough  an approximate solution of the following noisy transport equation
\bean \label{transport3}
2\nabla \gamma_\delta \cdot \nabla U^\delta +\gamma_\delta \Delta U^\delta = 0 \; \mathrm{in}\; \Omega
\;\; \mathrm{and}\;\; \gamma = \gamma_0,
\eean

Since  $\gamma_\delta = \gamma_0 $ and  $U^\delta$ are given   in $\Omega_\eta^C$, we can modify the behavior of $U^\delta$ in order to have a right hand term in $L^2(\Omega)$, and $\sigma_0\partial_\nu U^\delta = g$ on $\Gamma$, with  $g$ is chosen such that the associated inflow and outflow boundaries be well-separated, namely,
\bean \label{inout}
\mbox{dist}(\Gamma_{\textrm{in}}, \Gamma_{\textrm{out}}):=  \min_{(x,y)\in \Gamma_{\textrm{in}} \times \Gamma_{\textrm{out}} } |x-y|>0.
\eean
This condition is necessary and sufficient  to define traces of functions belonging to the graph space  of the  steady transport unbounded operator  in \cite{EGC}
\bea
L^2(|g|; \Gamma):= \left\{ v \textrm{  measurable on  }  \Gamma  :   \int_\Gamma |g| v^2 ds(x) <\infty  \right\}.
\eea
In order to ease the analysis we assume that the  right hand term  is zero. Notice that we can also pick $g = 0$ on $\Gamma$, and  in this particular case the boundary condition satisfied by $\gamma$ is not needed anymore. Since the transport equation \eqref{transport3} does not fall within the classical variational framework to prove the
existence, uniqueness of solutions, we introduce an auxiliary problem indexed by $\va>0$ which should be small enough. 

Fix $ \va \in (0, 1),$ and define the  regularized system corresponding to \eqref{transport3} as follows
\bean\label{transport5}
\beta \cdot\nabla \gamma +\mu_\va \gamma = 0\;\mbox{in}\;\Omega
\eean
where $\beta =\nabla U^\delta$ and $\mu_\va = \frac{1}{2} \Delta U^{\delta} +\va$. Due to the regularity of  $U^\delta \in H^5(\Omega)$,  the speed $\beta $ lies in $ C^{1,1}( \overline{\Omega})$, with $\|\nabla \beta_{i}\|_{L^\infty(\Omega)^2}\le L_{i}<\infty$ for $i=1,2$, $\beta_{,i}$ being the  components of $\beta$.
We deduce from \eqref{noise}  that the Lipschitz  constant $L :=\max_{i=1,2} L_{i}$ only depends on $g$, $\Sigma$ and $\Omega$.

Next we prove the existence and uniqueness of \eqref{transport5} with inflow boundary condition
\bea
\gamma = \gamma_0\; \mathrm{on}\; \Gamma_{\textrm{in}}.
\eea
Before that, we introduce the Graph space $V:=\{v\in L^2(\Omega)| \beta\cdot\nabla v\in L^2(\Omega)\}$ equipped with the natural scalar product
\bea
(v,w)_{V}:=(v,w)_{L^2(\Omega)}+(\beta\cdot\nabla v,\beta\cdot\nabla w)_{L^2(\Omega)},\quad\forall v,w\in V,
\eea
and the norm $\|v\|_{V}:=((v,v)_{V})^{1/2}$. It follows that $V$ is a Hilbert space and the triple $\{V,L^2(\Omega),V'\}$ is a Gelfand triple \cite{EGC}. Denote
\bea
L^2(|\beta\cdot n|,\Gamma):=\left\{v\;\mbox{is measurable on}\;\Gamma| \int_{\Gamma}|\beta\cdot \nu |v^2ds<\infty\right\}
\eea
the trace space. Lemma \ref{tracelemma} in the appendix allows us to define traces of functions belonging to the graph space $V$ and to use an integration by parts formula. In addition, for a real number $r$, we define its positive and negative parts respectively as
\bea
r^\oplus:=\frac{1}{2}\{|r|+r\},\quad x^\ominus:=\frac{1}{2}\{|r|-r\}.
\eea

We consider the weak form of \eqref{transport5} as follows:
\bean
\label{eq3}
\mbox{Find}\;\; \gamma_\va^\delta \in V \;\;\mbox{such that}\;\; a(\gamma_\va^\delta,w)= \int_{\Gamma} (\beta\cdot \nu)^\ominus
\gamma_0w\,ds \;\;\mbox{for all}\;\; w\in V,
\eean
where
\bea
a(v, w):=\int_\Omega\left(\beta\cdot\nabla v+\mu_\va v\right)w\,dx +\int_{\Gamma} (\beta\cdot \nu)^\ominus vw\,ds.
\eea

\begin{theorem} \label{existenceuniquenessregularized}
Let  $g$ be fixed in  $H^{\frac{7}{2}}(\Gamma)$ satisfying $\int_\Gamma g dx = 0$ and
condition \eqref{inout}. Then there exists a unique solution  $\gamma_{ \va}^\delta \in V$ to
the system  \eqref{transport5} for $\va \in (0, 1)$.
\end{theorem}
\begin{proof}
We follow the proof in \cite{PE12, EGC} and trace out the dependence of constant  in terms of the regularized parameter $\va$. The proof proceeds in four steps. Further $c>0$ denotes a generic constant that only depends on $\Sigma, \, g$ and $\Omega$.

We first prove that (\ref{eq3}) admits at most one solution. For all $v\in V$, we obtain from integration by parts that
\bean
\label{eq4}
a(v,v) &=& \va \int_\Omega v^2\,dx +\int_{\Gamma} \left[\frac{1}{2}(\beta\cdot \nu)+(\beta\cdot \nu)^\ominus\right]v^2\,ds\nonumber\\
&\ge& \|v\|_{L^2(\Omega)}^2 \va + \frac{1}{2}\int_{\Gamma} |g|v^2\,ds,
\eean
which implies the desired uniqueness. To prove the existence, we introduce an auxiliary  problem:
\bean
\label{eq5}
\mbox{Find}\;\; v'\in V \;\;\mbox{such that}\;\; a(v',w)=-a_0(\gamma_0,w) \;\;\mbox{for all}\;\; w\in V,
\eean
where
\bea
a_0(v,w):=\int_\Omega\left(\beta\cdot\nabla v+\mu_\va v\right)w\,dx.
\eea
The map $V\ni w\mapsto a_0(\gamma_0,w)\in\R$ is bounded in $L^2(\Omega)$.  Due to the  regularity of $\gamma_0$,
the function $f:= \beta\cdot\nabla \gamma_0+\mu_\va \gamma_0$ lies in  $L^2(\Omega) \subset V^\prime$.  Then  if   (\ref{eq5})
admits a solution $v'\in V$, we can obtain that $ \gamma^{\delta}_\va =v'+ \gamma_0\in V$ satisfies (\ref{eq3}) which gives the existence. 

Hence it remains to prove that (\ref{eq5}) is well-posed. The uniqueness of (\ref{eq5}) also follows from (\ref{eq4}). Set $V_{0}:=\{v\in V: v|_{\Gamma_{\textrm{in}}}=0\}$, and consider  the following problem:
\bean
\label{eq6}
\mbox{Find}\;\; v_0\in V_0 \;\;\mbox{such that};\; a_0(v_0,w)=\int_\Omega fw\,dx \;\;\mbox{for all}\;\; w\in L^2(\Omega).
\eean
If (\ref{eq6}) is well-posed and let $v_0\in V_0$ be its unique solution, we have that $v_0\in V$ and $a(v_0,w)=a_0(u,w)$ for all $w\in V$. It means that $v_0$ is a solution of (\ref{eq5}). It remains to prove that (\ref{eq6}) is well-posed.

Here we shall apply the Banach-Ne\v{c}as-Babu\v{s}ka (BNB) Theorem \ref{BNB} with $X=V_0$ and $Y=L^2(\Omega)$. In fact $V_0$ is a Hilbert space since $V_0$ is closed in $V$ and $L^2(\Omega)$ is a reflexive Banach space. The right-hand side in (\ref{eq6}) is a bounded linear form in $L^2(\Omega)$ and $a_0\in \mathcal{L}(V_0\times L^2(\Omega),\R)$ since
\bea
|a_0(v,w)| \le (1+\frac{1}{2}\|\Delta U^\delta
\|_{L^\infty(\Omega)}) \|v\|_{V}\|w\|_{L^2(\Omega)},\quad \forall (v,w)\in X\times Y.
\eea
It remains to verify the conditions (\ref{cond1}) and (\ref{cond2}) of the BNB Theorem.

(i). Proof of condition (\ref{cond1}). For $v\in V_0$, we set
\bea
\mathcal{S}_v:=\sup_{0\ne w\in L^2(\Omega)}\frac{a_0(v, w)}{\|w\|_{L^2(\Omega)}}.
\eea
Similar as (\ref{eq4}), we can get
\bea
a_0(v,v)\ge \va\|v\|_{L^2(\Omega)}^2\quad\mbox{for}\quad v\in V_0.
\eea
Then for $0\ne v\in V_0$,
\bea
\|v\|_{L^2(\Omega)}\le \frac{1}{\va} \frac{a_0(v,v)}{\|v\|_{L^2(\Omega)}}\le \frac{1}{\va}\mathcal{S}_v.
\eea
On the other hand,
\bea
\|\beta\cdot\nabla v\|_{L^2(\Omega)} &=& \sup_{0\ne w\in L^2(\Omega)} \frac{\int_\Omega (\beta\cdot\nabla v)w\,dx}{\|w\|_{L^2(\Omega)}}\\
&=& \sup_{0\ne w\in L^2(\Omega)} \frac{a_0(v,w)- \int_\Omega (\frac{1}{2}\Delta U^\delta+\va)  v w\,dx}{\|w\|_{L^2(\Omega)}}\\
&\le& \mathcal{S}_v+( \frac{1}{2}\|\Delta U^\delta\|_{L^\infty(\Omega)}+1)  \|v\|_{L^2(\Omega)}\\
&\le& \left(1+ ( \frac{1}{2}\|\Delta U^\delta\|_{L^\infty(\Omega)}+1) \va^{-1} \right) \mathcal{S}_v.
\eea
Therefore,
\bea
\|v\|_{V}^2=\|v\|_{L^2(\Omega)}^2 +\|\beta\cdot\nabla v\|_{L^2(\Omega)}^2&\le& \left[\va^{-2}+ \left(1+ ( \frac{1}{2}\|\Delta U^\delta\|_{L^\infty(\Omega)}+1) \va^{-1} \right)^2 \right]\mathcal{S}^2_v,\\
&\le&c \va^{-2} \mathcal{S}_v^2.
\eea
Then, we can take
\bea
C_{sta}:= c\va,
\eea
whence we infer condition (\ref{cond1}).

(ii).  Proof of condition (\ref{cond2}). Let $w\in L^2(\Omega)$ be such that $a_0(v,w)=0$ for all $v\in V_0$. Taking $v\in C_0^\infty(\Omega)$ first we get $\mu_\va w-\mbox{div}(\beta w)=0$ in $\Omega$. Hence, $\beta \cdot \nabla w=\mu_\va w-(\mbox{div}\beta)w\in L^2(\Omega)$ implying that $w\in V$. Then we have, for all $v\in V_0$
\bea
\int_{\Gamma} (\beta\cdot \nu)vw &=& \int_\Omega\left[(\beta\cdot\nabla v)w+(\beta\cdot\nabla w)v +(\nabla\cdot\beta)vw\right]dx\\
&=& a_0(v,w)\\
&=& 0.
\eea
Since $\Gamma_{\textrm{in}}$ and $\Gamma_{\textrm{out}}$ are  well-separated, there exist two functions $\psi^\pm\in C^\infty(\overline{\Omega})$  such that \cite{EGC}
\bea
\psi^{\textrm{out}}+\psi^{\textrm{in}}=1\quad\mbox{in}\quad\overline{\Omega},\quad \psi^{\textrm{in}}|_{\Gamma_{\textrm{out}}}=0,\quad \psi^{\textrm{out}}|_{\Gamma_{\textrm{in}}}=0.
\eea
Taking $v=\psi^{\textrm{out}}w\in V_0$, we obtain that
\bea
0=\int_{\Gamma} (\beta\cdot \nu)\psi^{\textrm{out}}w^2\,ds =\int_{\Gamma_{\textrm{in}}} (\beta\cdot \nu)w^2\,ds =\int_{\Gamma_{\textrm{out}}} (\beta\cdot \nu)^\oplus w^2\,ds,
\eea
which further implies that $w=0$ on $\Gamma_{\textrm{out}}$. 

Finally, we observe that
\begin{eqnarray*}
0&=& \int_\Omega [\mu_\va-\mbox{div}(\beta w)]w\,dx\\
&=& \va \int_\Omega
w^2\,dx + \frac{1}{2}\int_{\Gamma}(\beta\cdot \nu )w^2\,ds\\
&=& \va\|w\|_{L^2(\Omega)}^2.
\end{eqnarray*}
Therefore, $w=0$ in $\Omega$, which completes the proof.
\end{proof}

The unique solution of \eqref{transport} and the regularized one \eqref{transport5}, satisfy the  following stability estimate.

\begin{theorem} \label{theoremstability}
Let  $g$ be fixed in  $H^{\frac{7}{2}}(\Gamma)$ satisfying $\int_\Gamma g dx = 0$ and
condition \eqref{inout}. Then, there exist $c>0$ and  $s\in (0, \frac{1}{2})$, that only depends on $\sigma, g, \eta, $ and $ \Omega$ such that
\bean
\label{stability}
\int_{\Omega_\eta}|(\gamma_{\va}^\delta)^2- \gamma^2|^{1/2}\,dx \le c(1+\frac{\delta}{\va})(\va+\delta)^s,
\eean
where $\gamma $ and  $\gamma_{\va}^\delta $ are the solutions respectively to  \eqref{transport} and \eqref{transport5}. If in addition $|\nabla u_\sigma|$ does not vanish in $\Omega_\eta$, the inequality \eqref{stability} holds with  $s=\frac{1}{2}$.
\end{theorem}
\begin{proof}
Further $c>0$ denotes a generic constant that only depends on $\sigma, \, g, \,  \eta, $ and $\Omega$. We deduce from equations  \eqref{transport} and \eqref{transport5}
that $\xi = \gamma_{\va}^\delta- \gamma$ solves
\bea
\beta \cdot \nabla \xi  +\mu_\va \xi  =
\nabla (u_\sigma-U^\delta) \cdot  \nabla \gamma + \left(\frac{1}{2}\Delta(u_\sigma- U^\delta)-\va \right)
\gamma\; \mathrm{in}\; \Omega
\;\; \mathrm{and}\;\; \xi = 0\; \mathrm{on}\; \Gamma_{\textrm{in}}.
\eea
Using the variational formulation for the transport equation, and the fact that  $\sigma\in \Sigma$,
we find
\bea
\va  \|\xi\|_{L^2(\Omega)} \leq c\| U_\delta- u_\sigma\|_{H^1(\Omega)} +c\| U^\delta- u_\sigma\|_{H^2(\Omega)}+c\va.
\eea
Combining the inequality above with  estimate \eqref{noise}, we get
\bean \label{bbn1}
 \|\xi \|_{L^2(\Omega)} \leq c \left(\frac{\delta}{\va} +1 \right).
\eean
Since $\sigma\in \Sigma$, we immediately  deduce from \eqref{bbn1} the following bound
\bean \label{bbn2}
\|\gamma_\va^\delta\|_{L^2(\Omega)} \leq c\left(\frac{\delta}{\va} +1 \right).
\eean
Recall that $\gamma^2$  and $(\gamma_\va^\delta)^2$ solve respectively the following systems:
\bea
\mathrm{div}\left((\gamma_{\va}^\delta)^2 \nabla U^\delta\right)=
-2\va (\gamma_\va^\delta)^2\; \mathrm{in}\; \Omega \;\; \mathrm{and}\;\; (\gamma_\va^\delta)^2= \sigma_0\; \mathrm{on}\; \Gamma,
\eea
and
\bea
\mathrm{div}(\gamma^2 \nabla u_\sigma)=0\; \mathrm{in}\; \Omega \;\; \mathrm{and}\;\; \gamma^2= \sigma_0\; \mathrm{on}\; \Gamma,
\eea
Let $\mbox{sgn}_0$ be the sign function defined on $\R$ by: $\mbox{sgn}_0(t)=-1$ if $t<0$, $\mbox{sgn}_0(0)=0$ and $\mbox{sgn}_0(t)=1$ if $t>0$. Note that
\bea
\mbox{div}\left(|(\gamma_\va^\delta)^2- \gamma^2|\nabla U^\delta\right) &=& \mbox{sgn}_0((\gamma_\va^\delta)^2- \gamma^2) \mbox{div}\left((\gamma_\va^\delta)^2- \gamma^2)\nabla U^\delta\right)\\
&=& \mbox{sgn}_0((\gamma_\va^\delta)^2- \gamma^2) \left[\mbox{div}\left(\gamma^2\nabla(u_\sigma-U^\delta)\right)
-2\va (\gamma_\va^\delta)^2\right].
\eea
We get by integrating by parts that
\bea
\|\left((\gamma_\va^\delta)^2- \gamma^2\right)|\nabla U^\delta|^2\|_{L^1(\Omega)} \hspace{-2mm}&=& \hspace{-2mm}\int_\Omega \left|\mbox{div}(|(\gamma_\va^\delta)^2- \gamma^2|\nabla U^\delta)\right||U^\delta| \,dx \\
&\leq & \hspace{-2mm} \int_\Omega \left(\left|\mbox{div}\left(\gamma^2\nabla(u_\sigma-U^\delta)\right)\right|+2\va (\gamma_\va^\delta)^2\right)
 |U^\delta| \,dx.
\eea
Thus,
\bea
\|\left((\gamma_\va^\delta)^2- \gamma^2\right)|\nabla u_\sigma|^2\|_{L^1(\Omega)} &\le&
\|((\gamma_\va^\delta)^2- \gamma^2)(|\nabla u_\sigma|^2 -|\nabla U^\delta |^2)\|_{L^1(\Omega)}\\
&\quad& +c\|u_\sigma- U^\delta\|_{H^2(\Omega)}+c(\va+\delta).
\eea
We obtain that
\bean
\label{bbn3}
\||\left((\gamma_\va^\delta)^2- \gamma^2\right)|\nabla u_\sigma|^2\|_{L^1(\Omega)} \le c(\va+\delta)(1+\frac{\delta^2}{\va^2}).
\eean

Denote $\Omega^t_\eta:=\{x\in\Omega_\eta: |\nabla u_\sigma(x)|\geq t\}$. When $t$ tends to zero we expect $|\Omega_\eta \setminus \Omega^t_\eta|$ to approach  zero. The rate of decay depends on how  does $ |\nabla u_\sigma(x)|$ vanish at its critical points. The proof of this technical lemma is given in Appendix \ref{App}.

\begin{lemma} \label{Lemmaunique}
Let $\Omega^t_\eta:=\{x\in\Omega_\eta: |\nabla u_\sigma(x)| \geq t\}$. Then the following inequality holds
\bean \label{uniquecontinuation}
0 \leq |\Omega_\eta\setminus \overline{\Omega^t_\eta}| \leq c t^\alpha,
\eean
where $c>0$ and $\alpha >0 $ only depend
on $\sigma, \Omega, \eta$ and $g$.
\end{lemma}

Then, we get
\bean \label{xx1}
\int_{\Omega^t_\eta}|\gamma^2-(\gamma_\va^\delta)^2|^{1/2}\,dx &\le& t^{-1}\int_{\Omega^t_\eta}|\gamma^2-(\gamma_\va^\delta)^2|^{1/2}|\nabla u_\sigma|\,dx \\
&\le& t^{-1}\int_{\Omega}|\gamma^2-(\gamma_\va^\delta)^2|^{1/2}|\nabla u_\sigma|\,dx  \nonumber\\
&\le& t^{-1}|\Omega|^{1/2} \left(\int_{\Omega}|\gamma^2-(\gamma_\va^\delta)^2||\nabla u_\sigma|^2\,dx\right)^{1/2} \nonumber\\
&\le& ct^{-1}\left((\va+\delta)(1+\frac{\delta^2}{\va^2}) \right)^{1/2},\nonumber
\eean
and
\bea
\int_{\Omega_\eta\backslash \overline{\Omega^t_\eta}}|\gamma^2-(\gamma_\va^\delta)^2|^{1/2}\,dx &\le& |\Omega_\eta\backslash\Omega^t_\eta|^{1/2}\|\gamma^2-(\gamma_\va^\delta)^2\|_{L^1(\Omega)} \\
&\le& C|\Omega_\eta\backslash\Omega^t_\eta|^{1/2}\left(\frac{\delta}{\va} +1 \right)\\
&\le& ct^{\frac{\alpha}{2} } \left(\frac{\delta}{\va} +1 \right).
\eea
Hence
\bea
\int_{\Omega_\eta\backslash\overline{\Omega^t_\eta}}|\gamma^2-(\gamma_\va^\delta)^2|^{1/2}\,dx &\le& \int_{\Omega\backslash\Omega^t_\eta}|\gamma^2-(\gamma_\va^\delta)^2|^{1/2}\,dx+ +\int_{\Omega^t_\eta}|\gamma^2-(\gamma_\va^\delta)^2|^{1/2}\,dx\\
&\le& c\left [ t^{-1}\left((\va+\delta)(1+\frac{\delta^2}{\va^2}) \right)^{1/2}+  t^{\frac{\alpha}{2} } \left(\frac{\delta}{\va} +1 \right)\right ]\\
&\le& c\left [ t^{-1}(\va+\delta)^{1/2} +  t^{\frac{\alpha}{2} } \right ]\left(\frac{\delta}{\va} +1 \right).
\eea
Minimizing the right-hand side with respect to $t>0$, for fixed $\delta$ and $\va$ in $(0, 1)$, we find that the minimum is reached  at $t = (\va+\delta)^{\frac{1}{\alpha+2}}$, and verifies
\bea
 t^{-1}(\va+\delta)^{1/2}= t^{\frac{\alpha}{2}}.
\eea
Then, we get
\bean
\label{bbn4}
\int_{\Omega_\eta}|\gamma^2-(\gamma_\va^\delta)^2|^{1/2}\,dx \le c(1+\frac{\delta}{\va})(\va+\delta)^s,\quad s=
\frac{\alpha}{2(\alpha+2)},
\eean
which finishes the proof of the Theorem when  $\Omega_\eta\setminus \overline{\Omega^t_\eta}$ is not empty. 
\end{proof}

Assuming  now that  $|\nabla u_\sigma|$ does not vanish in $\Omega_\eta$. Regarding the regularity of $|\nabla u_\sigma|$, $\Omega_\eta\setminus \overline{\Omega^t_\eta}$ is empty for $t>0$ small enough, that is $\Omega_\eta =  \Omega^t_\eta$. We then deduce from  \eqref{xx1} that the  inequality \eqref{stability} is valid with $s= \frac{1}{2}$.

\begin{remark}
Notice that  $ \varphi \rightarrow \int_{\Omega} |\varphi|^{1/2} dx$ defines a complete metric on $L^{1/2}(\Omega)$. In fact  it is only  a quasi-norm since it does not satisfy the triangle inequality. Meanwhile the H\"older inequality  still holds \cite{AF}.
\end{remark}

Next, we study the  regularity of the unique weak solution  $\gamma_\va^\delta$ of the regularized transport equation \eqref{transport5}. The main difficulty of the theory of the boundary value problem for the transport equation in the case of nonempty set
$\partial \Gamma_{\textrm{in}}$ is that a solution may develop singularities at $ \partial \Gamma_{\textrm{in}}$ (See for instance  example 12.2.1 in \cite{PS}).

\begin{theorem}  Let  $g$ be fixed in  $H^{\frac{7}{2}}(\Gamma)$ satisfying $\int_\Gamma g dx = 0$ and condition \eqref{inout}, and let   $\gamma_{ \va}^\delta \in V$ be the unique solution to \eqref{transport5} for $\va \in (0, 1)$. Assume in addition that $u_\sigma$ is a convex function. Then there exists $c_0>0$ that only depends on  $\Sigma, \, g$ and $\Omega$ such that  if $\frac{\delta}{\va} \in (0,  c_0)$,  $\gamma_{ \va}^\delta$ lies in $H^1(\Omega)$, and it satisfies
\bean
\label{H1estimate}
\| \gamma_\va^\delta\|_{H^1(\Omega)}\le c\va^{-1},
\eean
where $c>0$ only depends on $\Sigma, \, g$ and $\Omega$.
\end{theorem}
\begin{proof}
Differentiating the regularized equation \eqref{transport5}, we obtain that $ \zeta= \nabla \gamma_\va^\delta$ satisfies the following Friedirich system
\bea
(\partial _k U^\delta I_n) \partial_k \zeta
+\left( (\frac{1}{2} \Delta U^\delta+\va )I_n+\mathcal H(U^\delta) \right) \zeta =
-\frac{1}{2} \gamma_\va^\delta \nabla \Delta U^\delta,
\eea
$\partial_i $ denotes the partial derivative with respect to $x_i$, $I_n$ is identity matrix in $\mathbb R^n$,
$H(U_\va^\delta)$  is the Hessian matrix of $U_\va^\delta$.

Since the function $u_\sigma$ is convex, we deduce from \eqref{noise} that there exists a constant $c_0>0$, that only depends on  $\Sigma, \, g$ and $\Omega$, such that
\bea
\mathcal H(U^\delta)  \geq -\frac{ \delta}{2c_0} I_n.
\eea
Hence for $\frac{\delta}{\va} \in (0,  c_0)$, we have
\bea
\va I_n +\mathcal H(U^\delta) \geq \frac{\va}{2}.
\eea

The variational necessary condition for the existence and uniqueness of solution to the Friedirich   system is then satisfied, and we have \cite{EGC}
\bea
\| \zeta \|_{\left(L^2(\Omega)\right)^n} \leq c\va^{-1}\|\gamma_\va^\delta\|_{L^2(\Omega)}.
\eea
Combining the previous estimate with  \eqref{bbn2}, gives the desired result.
\end{proof}
\begin{remark}
Note that the assumption on the convexity of  $u_\sigma$ immediately  implies $\Delta u_\sigma \geq 0$ since $\Delta u_\sigma = \textrm{tr}(\mathcal H(u_\sigma)).$ 
\end{remark}

\section{Discontinuous Galerkin method}
\label{sec:3}
In this section, we discretize the system \eqref{eq3} by a discontinuous Galerkin (DG) method~\cite{PE12}. Let $T=\{T_h\}_h$ be a family of conforming quasi-uniform triangulations such that $\overline{\Omega}=\cup_{\tau\in T_h}\overline{\tau}$, $\tau_i\cap\tau_j=\emptyset$ for $\tau_i,\tau_j\in T_h$, $i\ne j$. Set $h_\tau=\mbox{diam}(\tau)$ and $h=\max_{\tau\in T_h} h_\tau$. For an integral $k$ and $\tau\in T_h$, let $\mathbb{P}^k(\tau)$ be the set of all polynomials on $\tau$ of degree
at most $k$. We define the discrete space
\bea
V_h:=\{v\in L^2(\Omega): v|_\tau\in \mathbb{P}^k(\tau)\quad\forall\;\tau\in T_h\}.
\eea
We split the set of all edges $\mathcal{E}_h$ into the set $\mathcal{E}_h^i$ of interior edges of $T_h$ and the set $\mathcal{E}_h^\partial$ of boundary edges of $T_h$ such that $\mathcal{E}_h=\mathcal{E}_h^i\cup\mathcal{E}_h^\partial$. For an $e\in\mathcal{E}_h^i$, we define the averages and jumps of $v\in V_h$ by
\bea
[v]_e=\lim_{\rho \rightarrow 0^+}\left[v(x- \rho n_e)-v(x+\rho n_e)\right],
\eea
and
\bea
\{v\}_e= \frac{1}{2} \lim_{\rho \rightarrow 0^+}\left[v(x-\rho n_e)+v(x+\rho n_e)\right],
\eea
respectively, where $n_e$ is one of the normal unit vectors to $e$. For $e\in\mathcal{E}_h^\partial$, $n_e$ denotes the outward unit normal.

We consider the discrete problem:
\bean
\label{DGdiscrete}
\mbox{Find}\; \gamma_{\va,h}\in V_h\;\mbox{such that}\; a_h(\gamma_{\va,h}^\delta,w_h)= \sum_{e\in\mathcal{E}_h^\partial}\int_e (\beta\cdot n_e)^\ominus f'w_h\,ds \;\;\mbox{for all}\;\; w_h\in V_h,
\eean
where the the upwind DG bilinear form $a_h$ is given by
\bea
a_h(\gamma_{\va,h}^\delta,w_h)&:=& \sum_{\tau\in T_h}\int_\tau \left[\mu_\va \gamma_{\va,h}^\delta w_h+(\beta\cdot \gamma_{\va,h}^\delta)w_h\right]\,dx+ \sum_{e\in\mathcal{E}_h^\partial} \int_e (\beta\cdot n_e)^\ominus \gamma_{\va,h}^\delta w_h\,ds\\
&\quad& -\sum_{e\in\mathcal{E}_h^i}\int_e (\beta\cdot n_e)[\gamma_{\va,h}^\delta]\{w_h\}\,ds+\sum_{e\in\mathcal{E}_h^i}\int_e \eta|\beta\cdot n_e|[\gamma_{\va,h}^\delta][w_h]\,ds.
\eea

In the following, we assume that $\va>0$ is sufficiently small. We first examine the consistency and discrete coercivity of the upwind DG bilinear form $a_h$. Assume that there is a partition $P_\Omega$ of $\Omega$ into disjoint polyhedra such that $\gamma_\va^\delta\in V_*:=V\cap H^1(P_\Omega)$. We set $V_{*h}=V_*+V_h$. This assumption implies that $\gamma_\va^\delta\in L^2(e)$ for all $e\in\mathcal{E}_h$. The space $H^1(P_\Omega)$ can be replaced by $H^{1/2+\epsilon}(P_\Omega)$ or $W^{1,1}(P_\Omega)$ for a weakly regularity assumption where $0<\epsilon<1/2$. From Lemma 2.14 in~\cite{PE12}, we know that for $u\in V_*$ and all $e\in\mathcal{E}_h^i$, $(\beta\cdot n_e)[u]=0$ a.e. on $e$. Define $\tau_c=\max\{\|\mu_{\va}\|_{L^\infty(\Omega)},L_{\beta}\}$, $\beta_c=\|\beta\|_{L^\infty(\Omega)^2}$ and the following strong norms
\bea
\normmm{v}^2:= \tau_c^{-1}\sum_{\tau\in T_h} \|v\|_{L^2(\tau)}+\frac{1}{2}\sum_{e\in\mathcal{E}_h^\partial} \int_e |\beta\cdot n_e| v^2\,ds+ \sum_{e\in\mathcal{E}_h^i}\int_e \eta|\beta\cdot n_e|[v]^2\,ds,
\eea
and
\bea
\normmm{v}_*^2:= \normmm{v}^2+\beta_c\sum_{\tau\in T_h}\|v\|_{L^2(\partial \tau)}^2
\eea
for $v\in V_{*h}$. For the consistency of $a_h$ we refer to Lemma 2.27(i) in~\cite{PE12} and we conclude the coercivity in the following lemma.
\begin{lemma}
For all $w_h\in V_h$, there exists a constant $C_c$ independent of $h,\va,w_h$ such that
\bean
\label{dgcoe}
a_h(w_h,w_h)\ge C_c\va\normmm{w_h}^2.
\eean
\end{lemma}
\begin{proof}
It follows from the coercivity of the bilinear form $a$ and the fact that $(\beta\cdot n_e)[w_h]$ vanishes across interior interfaces that
\bea
a_h(w_h,w_h)\ge \min\{1,c\tau_c\va/4\}\normmm{w_h}^2,
\eea
which yields the desired result since $\va$ is sufficiently small.
\end{proof}

The discrete coercivity of $a_h$ on $V_h$ implies the well-posedness of the discrete problem \eqref{DGdiscrete}. Let $\pi_h$ be the $L^2$-orthogonal projection operator onto $V_h$. Then we have (see Theorem 2.30 in \cite{PE12})
\begin{lemma}
For $(v,w_h)\in V_*\times V_h$, there exists a constants $C_b>0$ independent of $h,\va,v,w_h$ such that
\bean
\label{dgbou}
|a_h(v-\pi_hv,w_h)|\le C_b\normmm{v-\pi_hv}_*\normmm{w_h}.
\eean
\end{lemma}

Now we state the main result of error estimates.

\begin{theorem} \label{thmerror}
Let $\gamma_\va^\delta,\gamma_{\va,h}^\delta$ be the unique solution of \eqref{eq3} and \eqref{DGdiscrete}, respectively. Assume that $\gamma_\va\in H^{k+1}(\Omega)$. Then we have
\bean
\label{error1}
\normmm{\gamma_\va^\delta-\gamma_{\va,h}^\delta}\le C\va^{-1}h^{k+1/2} \|\gamma_\va^\delta\|_{H^{k+1}(\Omega)}.
\eean
Moreover, there exists a constant $0<s<1/2$ such that
\bean
\label{error2}
\int_\Omega|\gamma^\delta-\gamma_{\va,h}^\delta|^{1/2}\,dx \le C\left(\va^{-1}h^{k+1/2}\|\gamma_\va\|_{H^{k+1}(\Omega)}+ (1+\frac{\delta}{\va})(\va+\delta)^s\right),
\eean
Here, $C>0$ is a constant independent of $\gamma^\delta,\gamma_\va^\delta,h,\va,\delta$.
\end{theorem}
\begin{proof}
Following Theorem 2.31 in \cite{PE12} we know
\bea
\normmm{\gamma_\va^\delta-\gamma_{\va,h}^\delta}\le (1+C_1\va^{-1}) \normmm{\gamma_\va^\delta-\gamma_{\va,h}^\delta}_* \le C\va^{-1}\normmm{\gamma_\va^\delta-\gamma_{\va,h}^\delta}_*.
\eea
Using Lemma 1.58 and 1.59 in \cite{PE12} we obtain that
\bea
\normmm{\gamma_\va^\delta-\gamma_{\va,h}^\delta}_*\le Ch^{k+1/2}\|\gamma_\va^\delta\|_{H^{k+1}(\Omega)},
\eea
which implies \eqref{error1}. The estimate \eqref{error2} follows from \eqref{stability}, \eqref{H1estimate}, \eqref{error1} and the triangle inequality.
\end{proof}

\section{Numerical examples}

In this section, we present several numerical examples for the reconstruction of $\sigma$
that demonstrate the accuracy and efficiency of the proposed inversion algorithm. Although we assume $\Omega$ being a $C^6$-smooth bounded domain in theoretical analysis,
here $\Omega$ is set to be $[0,1]\times [0,1]$ for simplicity.  Since only $\nabla u_\sigma$ is used in the inversion
we do not impose the condition $\int_{\Omega} u_\sigma dx= 0$ in the rest of this section.  All of the numerical tests were obtained by means of Matlab numerical implementations. We always choose the parameter $\eta=100$ used in discontinuous Galerkin method described in Section~\ref{sec:3}. The numerical errors Error and the relative $L^2$-error RError are calculated in accordance with the expressions
\bea
\mbox{Error}=\int_\Omega |\gamma-\gamma_{\va,h}|^{1/2}\,dx,
\eea
and
\bea
\mbox{RError}=\frac{\|\gamma- \gamma_{\va,h}\|_{L^2(\Omega)}}{\|\gamma\|_{L^2(\Omega)}},
\eea
respectively.

{\bf Example 1.} Let the exact $u_\sigma$ in $\Omega$ be given by
\bea
u_\sigma=e^{0.5-x_1+(x_2-0.5)^2},
\eea
and the exact conductivity is
\bea
\sigma=e^{3x_1-0.5-(x_2-0.5)^2},
\eea
see Figure~\ref{Exp1.1}. The triangular partition of $\Omega$ is fixed with meshsize $h=0.0295$. The numerical errors for different $k$ and $\va$ are presented in Figure~\ref{Exp1.2}. For fixed small $\va$, the numerical errors, dominated by the errors arising from discontinuous Galerkin approximation, decay as $k$ increase. On the other hand, for fixed large $k$, the numerical errors are dominated by the term $\va^s$ with $s=\frac{\alpha}{2(\alpha+2)}$ given in (\ref{bbn4}). The results demonstrate the convergence of numerical errors with respect to $\va$ with order nearly $O(\va^{1/2})$ which corresponds to the exact theoretical rate since $|\nabla u_\sigma| $ does not vanish here
(Theorem~\ref{theoremstability}). The numerical reconstructions for $k=3$ are shown in Figure~\ref{Exp1.3} with the corresponding relative $L^2$-errors for different choices of $\va$.

\begin{figure}[htb]
\centering
\begin{tabular}{cc}
\includegraphics[scale=0.15]{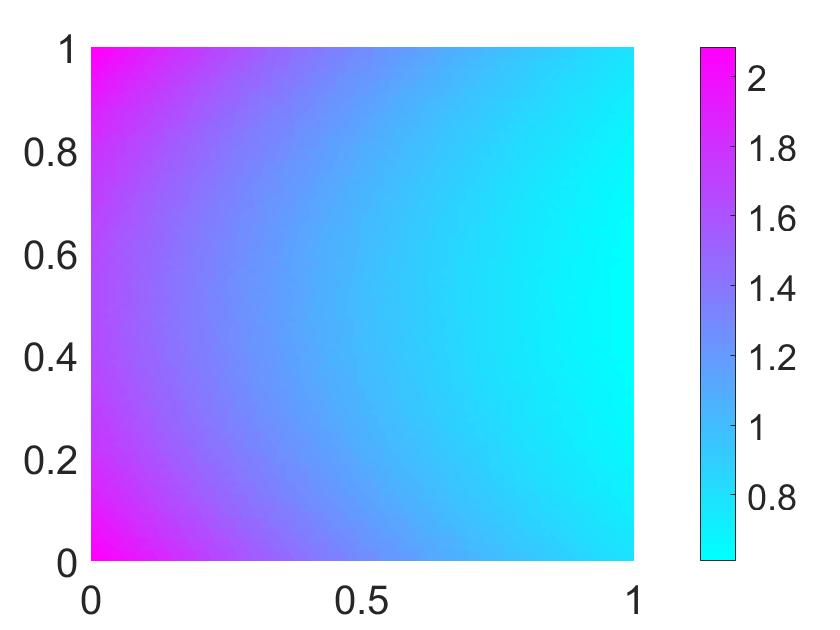} &
\includegraphics[scale=0.15]{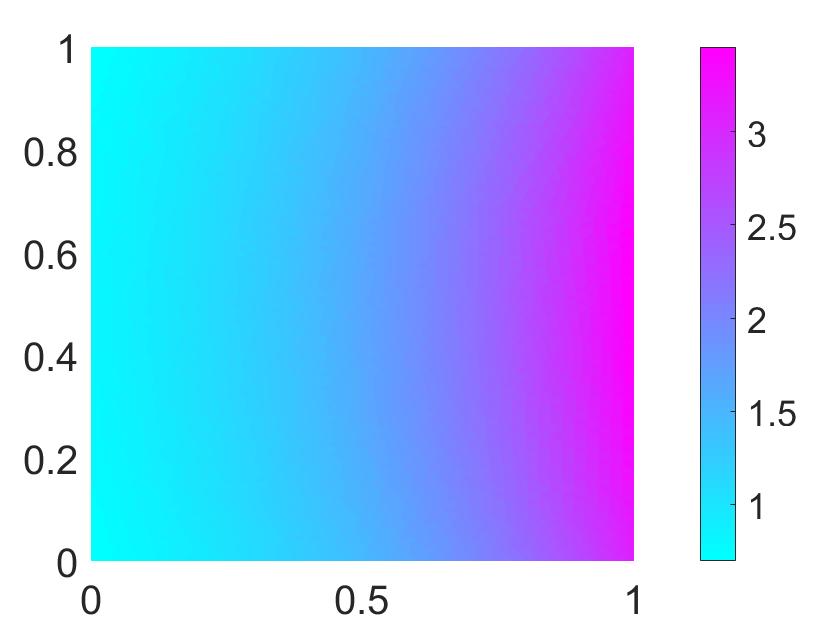} \\
(a) $u_\sigma$ & (b) $\gamma$ \\
\end{tabular}
\caption{Example 1. The exact $u_\sigma$ (a) and $\gamma$ (b).}
\label{Exp1.1}
\end{figure}

\begin{figure}[htb]
\centering
\includegraphics[scale=0.4]{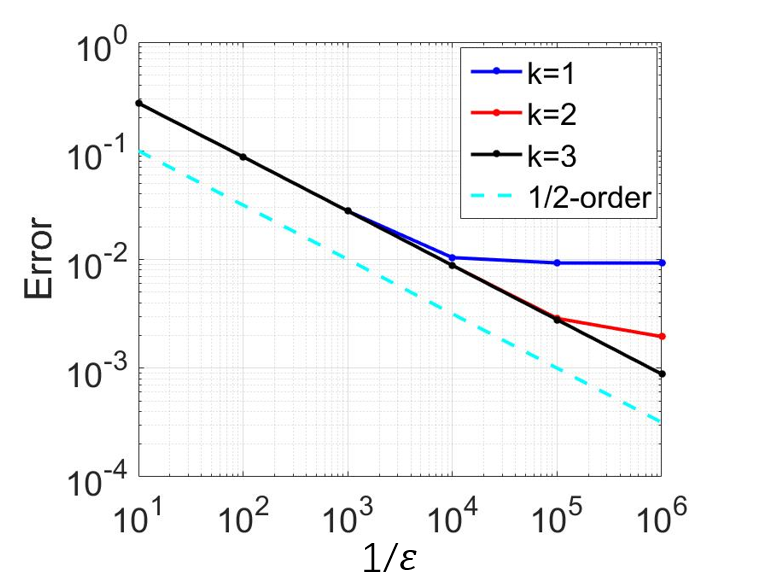}
\caption{Example 1. Numerical errors of the reconstruction for different $k$ and $\va$.}
\label{Exp1.2}
\end{figure}

\begin{figure}[htb]
\centering
\begin{tabular}{ccc}
\includegraphics[scale=0.12]{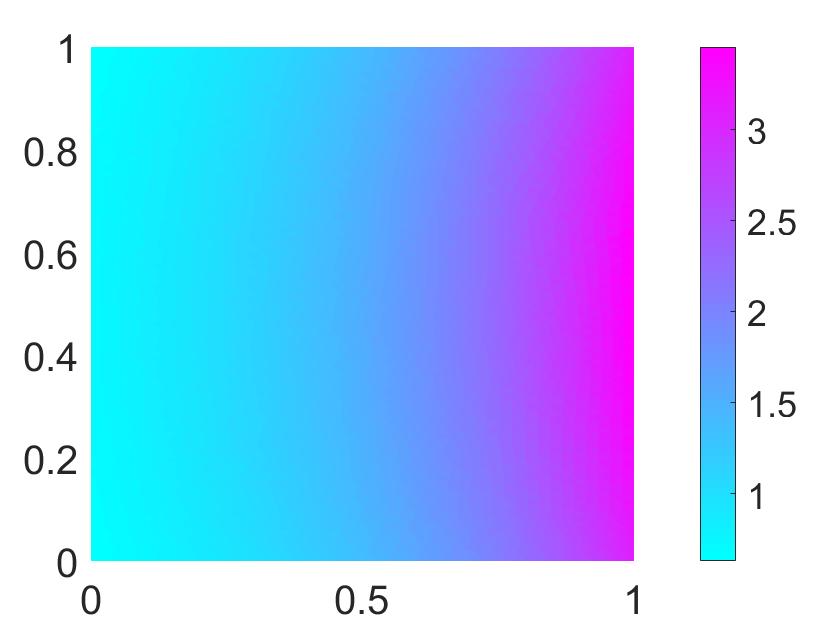} &
\includegraphics[scale=0.12]{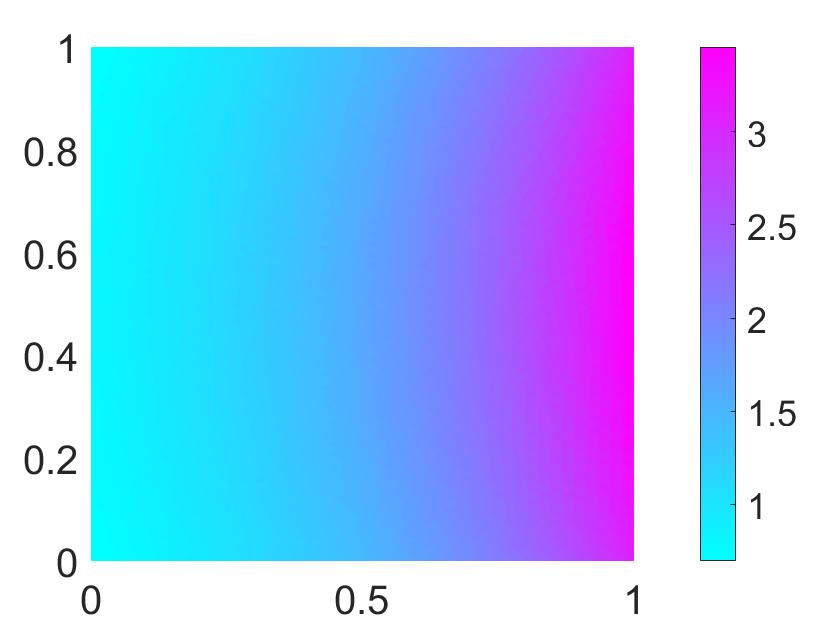} &
\includegraphics[scale=0.12]{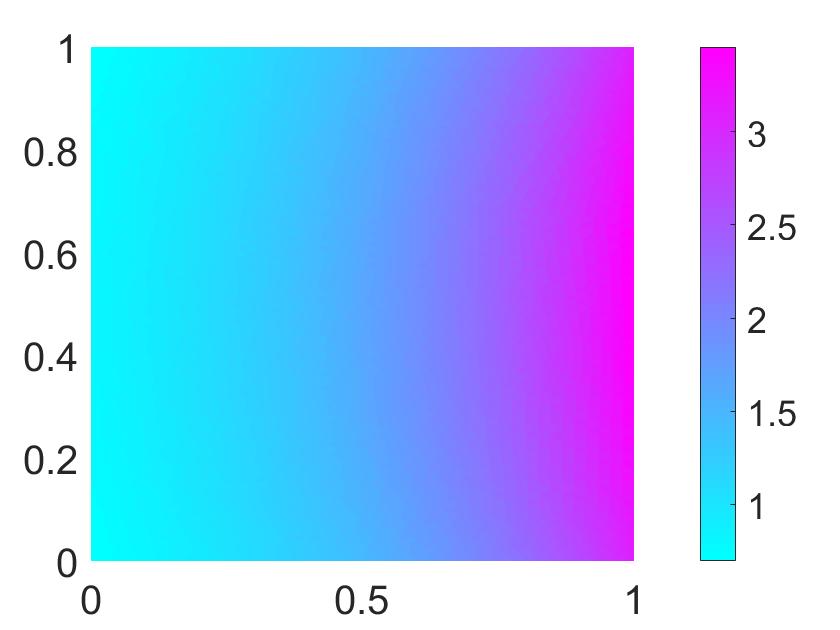} \\
(a) $\va=10^{-1}$ & (b) $\va=10^{-3}$ & (c) $\va=10^{-5}$
\end{tabular}
\caption{Example 1. The reconstruction of $\gamma$ for different $\va$ with relative errors RError=$4.31\times 10^{-2}$, $4.45\times 10^{-4}$ and $4.46\times 10^{-6}$ in (a,b,c), respectively.}
\label{Exp1.3}
\end{figure}

{\bf Example 2.} In this example, we consider the reconstruction of the conductivity function given by
\bea
\sigma(x_1,x_2)=q(3(2x_1-1),3(2x_2-1)),
\eea
where
\bea
q(x_1,x_2)&=&1+0.3(1-x_1)^2 e^{-x_1^2-(x_2+1)^2}-(\frac{1}{5}x_1-x_1^3-x_2^5)e^{-x_1^2-x_2^2}\\
&\quad& - \frac{1}{30}e^{-(x_1+1)^2-x_2^2},
\eea
see Figure~\ref{Exp2.1}. Set $g=e^{x_1+x_2}-(e^2-1)/2$ on $\Gamma$ satisfying $\int_\Gamma g\,dx=0$. The solution $u_\sigma$ and $|\nabla u_\sigma|$ are presented in Figure~\ref{Exp2.1.1}. For simplicity, let the exact measurement data be the polynomial coefficients of the discontinuous Galerkin approximation of $u_\sigma$ in $P^{k_0}(T_h), k_0\ge 2$. In other words, in each triangular element, we have at least $\frac{(k_0+1)(k_0+2)}{2}$ measurement points to produce the measurement data. The triangular partition of $\Omega$ is fixed with meshsize $h=0.0295$. Choosing $k_0=1$ will result into poor reconstruction since in this case $\Delta u_\sigma=0$ in each element. We choose the parameter $k=k_0$. The reconstruction results are presented in Figures~\ref{Exp2.2} and \ref{Exp2.3} and the corresponding numerical errors are displayed in Figure~\ref{Exp2.4}. The accuracy limitation at a level of approximately $10^{-2}$ corresponds to the number of measurement points and the approximations of $\nabla u_\sigma$ and $\Delta u_\sigma$.

\begin{figure}[htb]
\centering
\includegraphics[scale=0.15]{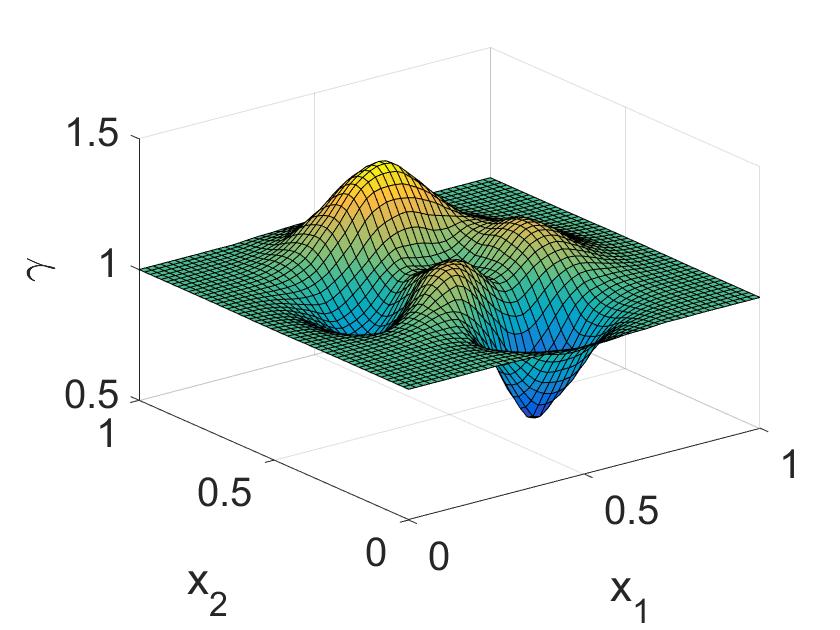}
\caption{Example 2. The exact $\gamma$.}
\label{Exp2.1}
\end{figure}

\begin{figure}[htb]
\centering
\begin{tabular}{cc}
\includegraphics[scale=0.15]{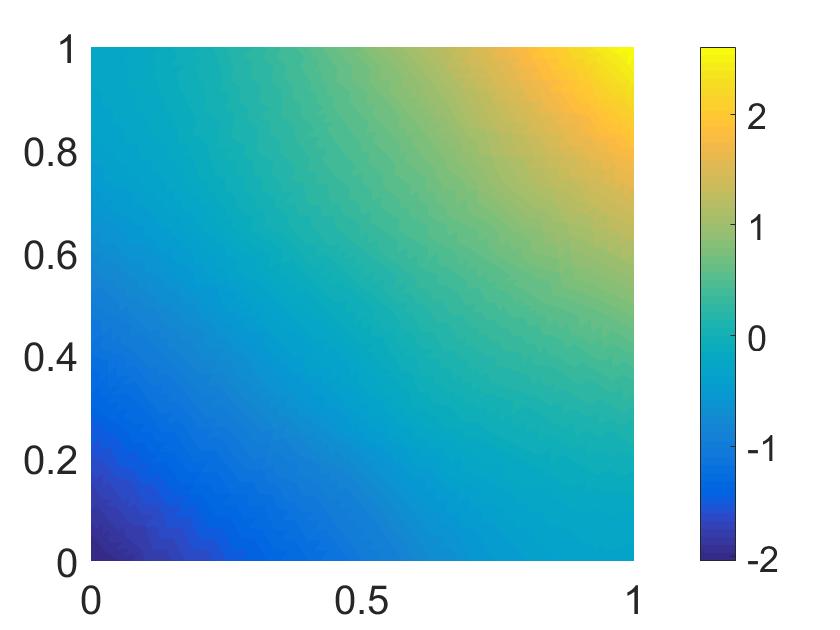} &
\includegraphics[scale=0.15]{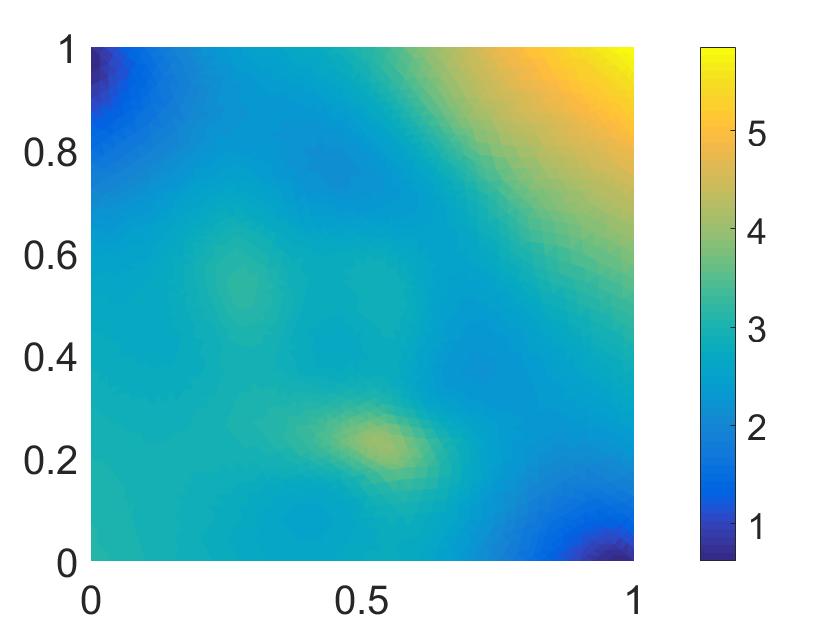} \\
(a) $u_\sigma$ &
(c) $|\nabla u_\sigma|$
\end{tabular}
\caption{Example 2. Exact data $u_\sigma$ and $|\nabla u_\sigma|$.}
\label{Exp2.1.1}
\end{figure}

\begin{figure}[htb]
\centering
\begin{tabular}{ccc}
\includegraphics[scale=0.12]{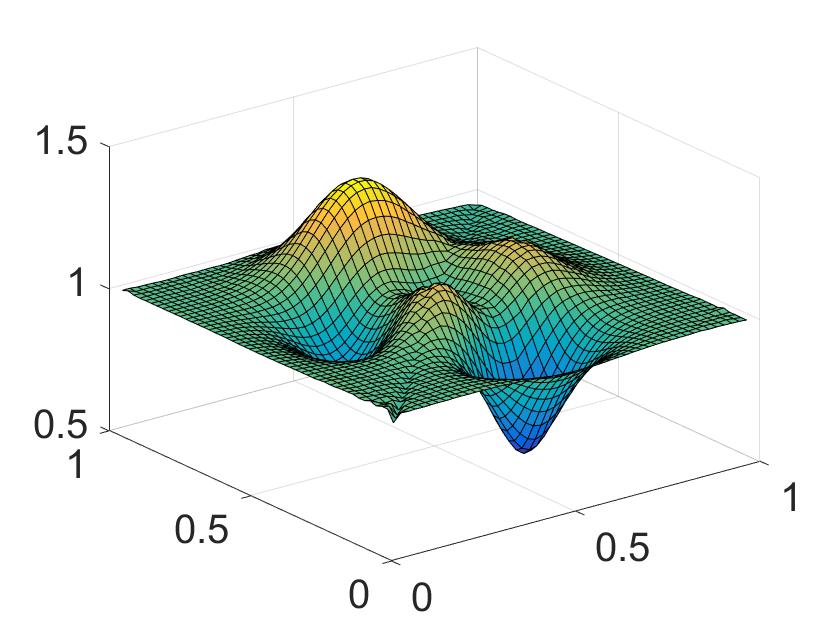} &
\includegraphics[scale=0.12]{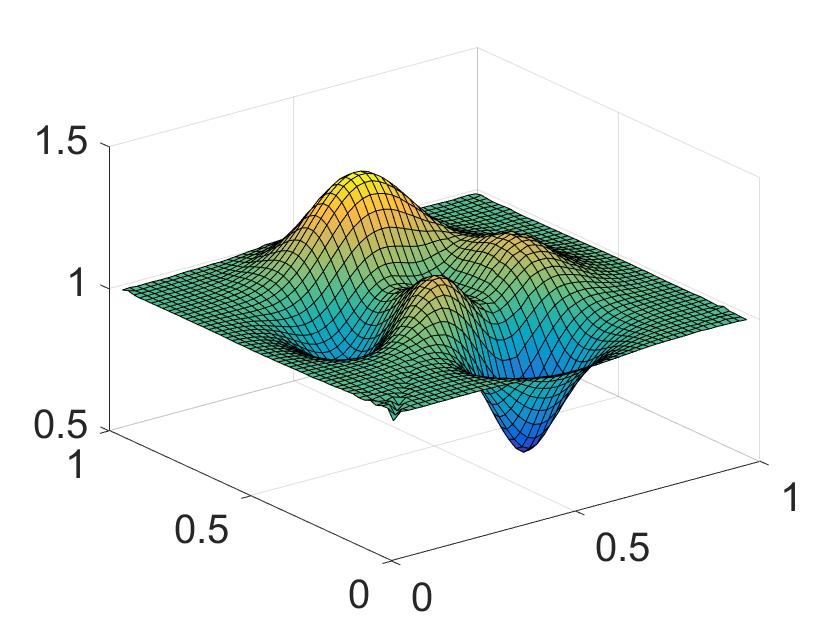} &
\includegraphics[scale=0.12]{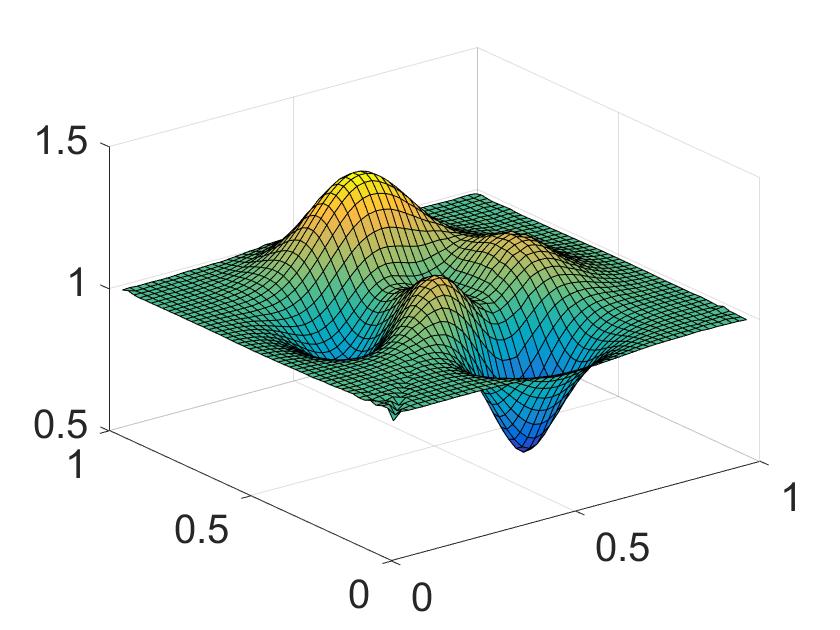} \\
(a) $\va=10^{-1}$ &
(b) $\va=10^{-3}$ &
(c) $\va=10^{-5}$
\end{tabular}
\caption{Example 2. The reconstruction of $\gamma$ when $k_0=2$ for different $\va$ with relative errors RError=$1.71\times 10^{-2}$, $2.53\times 10^{-3}$ and $2.50\times 10^{-3}$ in (a,b,c), respectively.}
\label{Exp2.2}
\end{figure}

\begin{figure}[htb]
\centering
\begin{tabular}{ccc}
\includegraphics[scale=0.12]{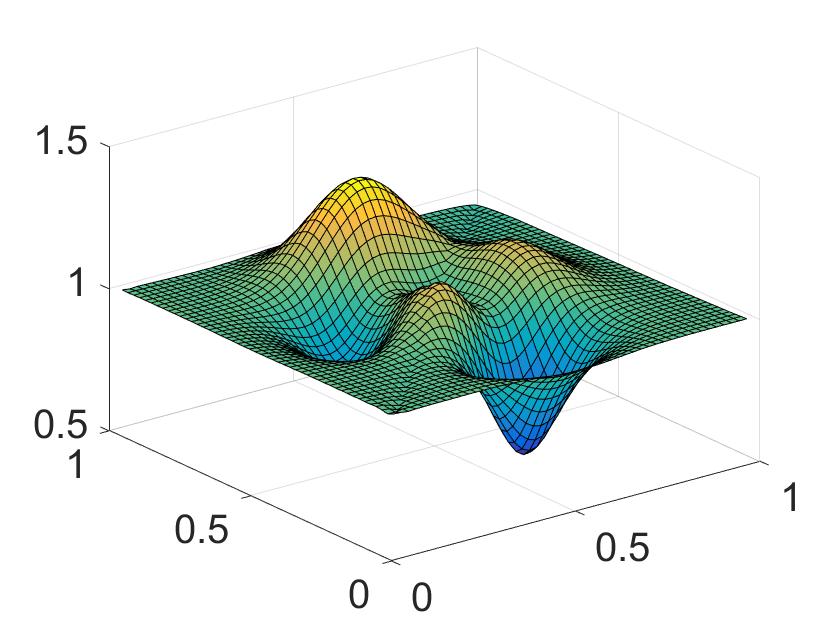} &
\includegraphics[scale=0.12]{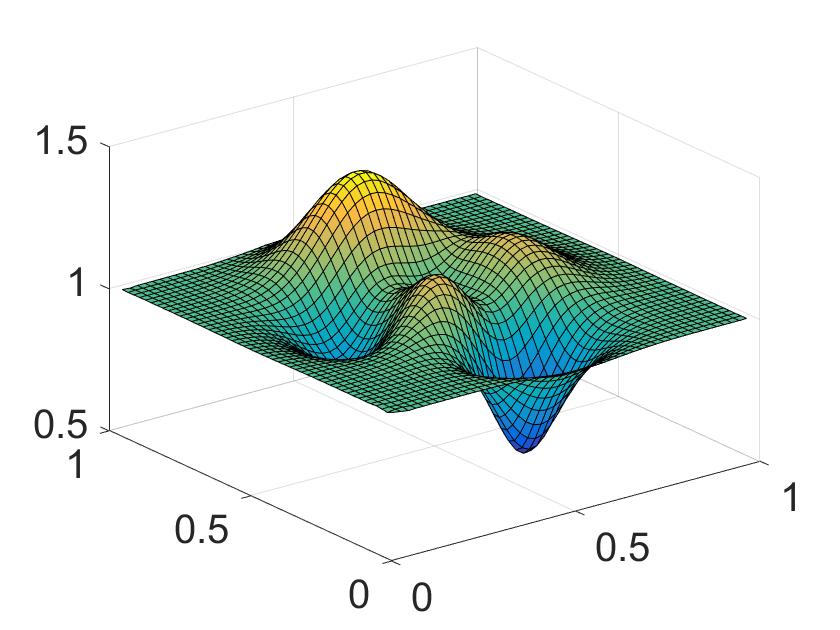} &
\includegraphics[scale=0.12]{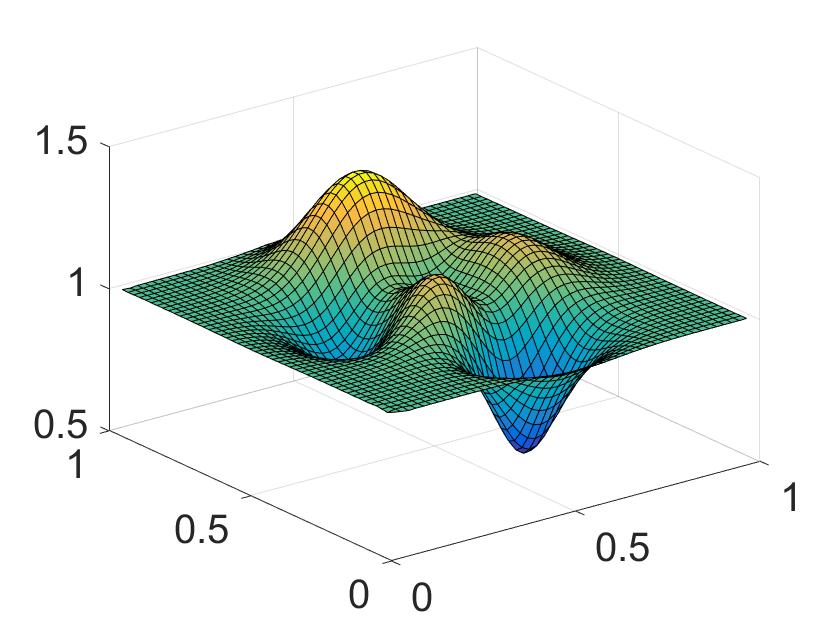} \\
(a) $\va=10^{-1}$ &
(b) $\va=10^{-3}$ &
(c) $\va=10^{-5}$
\end{tabular}
\caption{Example 2. The reconstruction of $\gamma$ when $k_0=3$ for different $\va$ with relative errors RError=$1.66\times 10^{-2}$, $4.04\times 10^{-4}$ and $3.64\times 10^{-4}$ in (a,b,c), respectively.}
\label{Exp2.3}
\end{figure}

\begin{figure}[htb]
\centering
\includegraphics[scale=0.4]{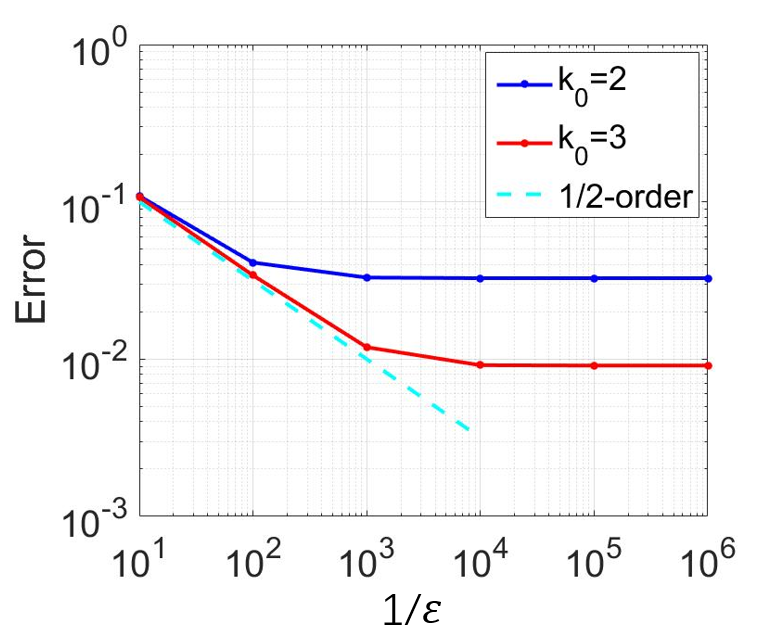}
\caption{Example 2. Numerical errors of the reconstruction for different $k_0$ and $\va$.}
\label{Exp2.4}
\end{figure}

{\bf Example 3.} Next, we consider the reconstruction of conductivity function discussed in Example 2 from noised data
\bea
U^\delta=u_\sigma(1+\delta\xi),
\eea
where $\delta$ is the noise level and $\xi$ is an independent and uniformly distributed random variable generated between -1 and 1. Figure~\ref{Exp3.1.1} displays the perturbated data $U^\delta$ and $\nabla U^\delta$ with $\delta=10\%$ random noise. In addition, less point measurements are taken under a triangular partition of $\Omega$ with meshsize $h=0.0589$ and we choose $k_0=2$. The reconstruction results from noised measurement with level $\delta=5\%$ are shown in Figures~\ref{Exp3.1} and \ref{Exp3.2} which demonstrate the high efficiency and robustness of the proposed inversion algorithm. Figure~\ref{Exp3.3} displays the convergence of numerical errors with respect to $\delta+\va$.

\begin{figure}[htb]
\centering
\begin{tabular}{cc}
\includegraphics[scale=0.15]{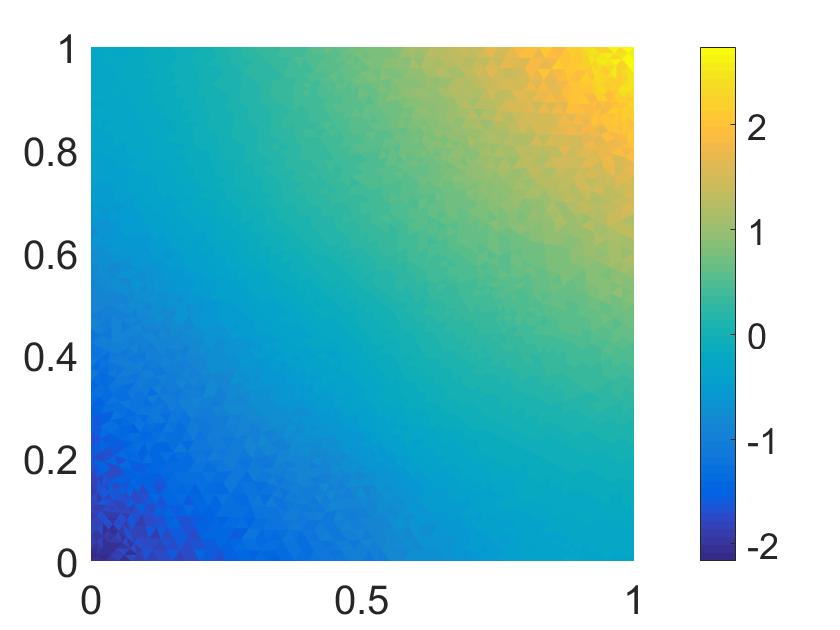} &
\includegraphics[scale=0.15]{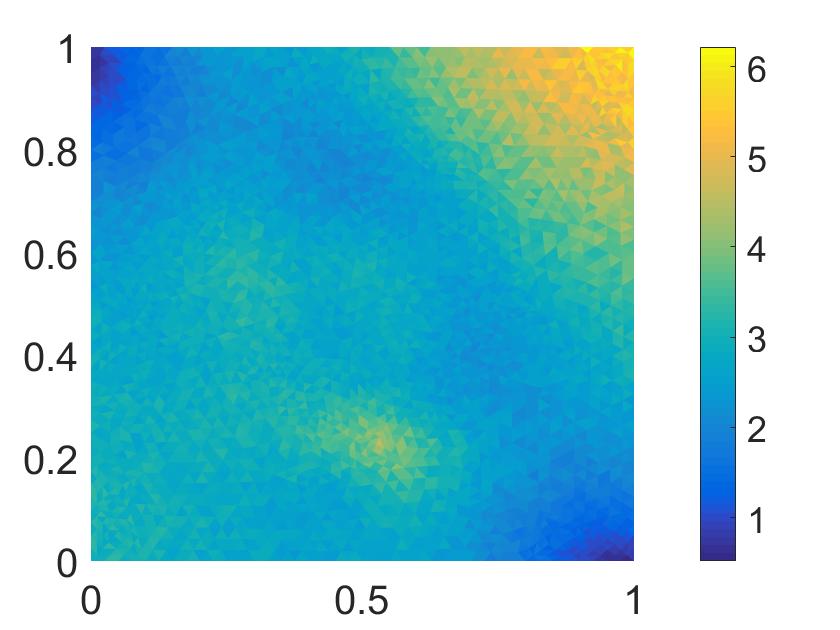} \\
(a) $U^\delta$ &
(c) $|\nabla U^\delta|$
\end{tabular}
\caption{Example 3. Perturbated data data $U^\delta$ and $|\nabla U^\delta|$ with $10\%$ random noise.}
\label{Exp3.1.1}
\end{figure}

\begin{figure}[htb]
\centering
\begin{tabular}{ccc}
\includegraphics[scale=0.12]{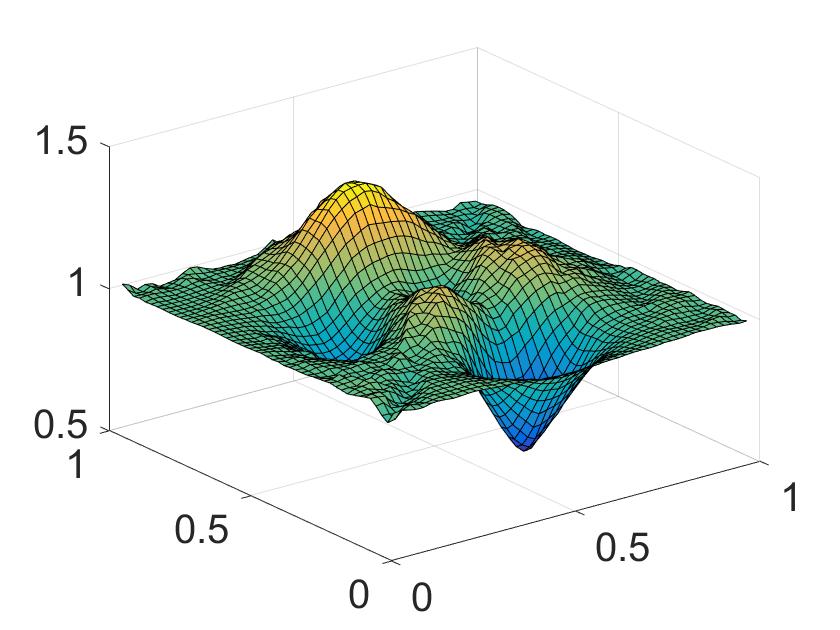} &
\includegraphics[scale=0.12]{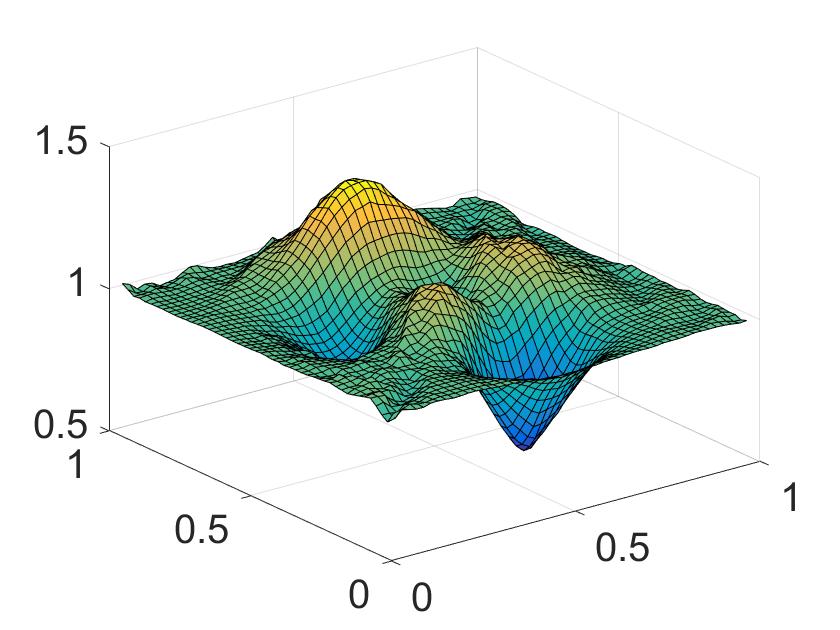} &
\includegraphics[scale=0.12]{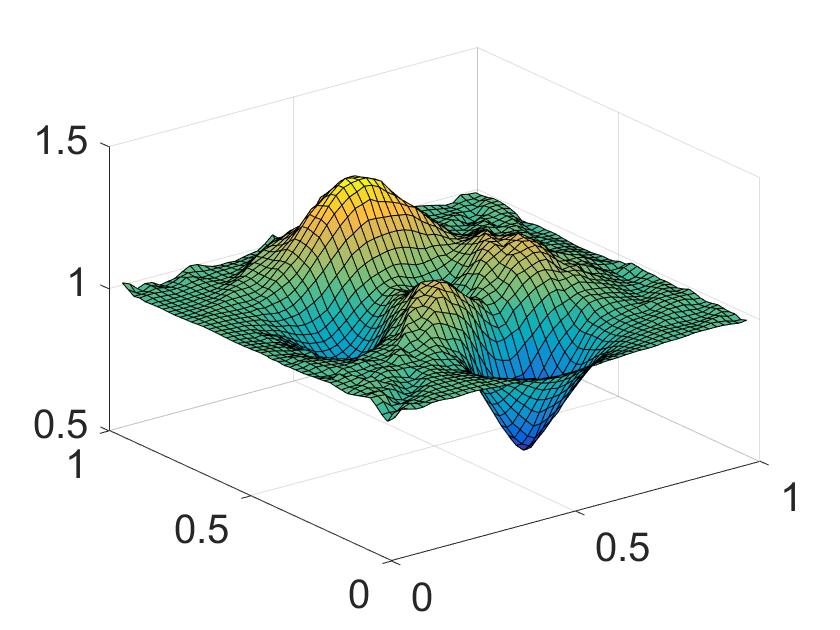} \\
(a) $\va=0.1$ &
(b) $\va=0.06$ &
(c) $\va=0.01$
\end{tabular}
\caption{Example 3. The reconstruction of $\gamma$ from noised measurements when $k_0=2$, $\delta=5\%$ for different $\va$ with relative errors RError=$2.24\times 10^{-2}$, $1.71\times 10^{-2}$ and $1.22\times 10^{-2}$ in (a,b,c), respectively.}
\label{Exp3.1}
\end{figure}

\begin{figure}[htb]
\centering
\begin{tabular}{ccc}
\includegraphics[scale=0.12]{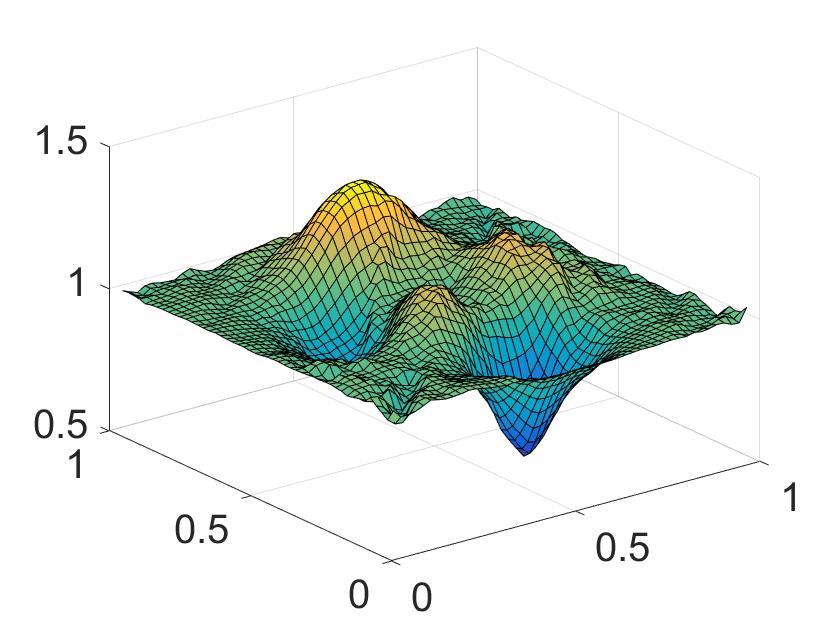} &
\includegraphics[scale=0.12]{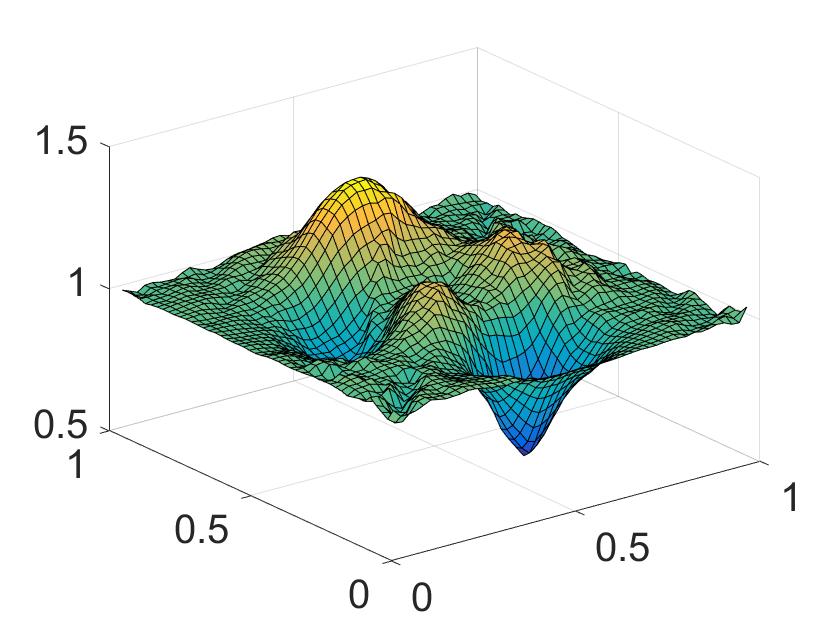} &
\includegraphics[scale=0.12]{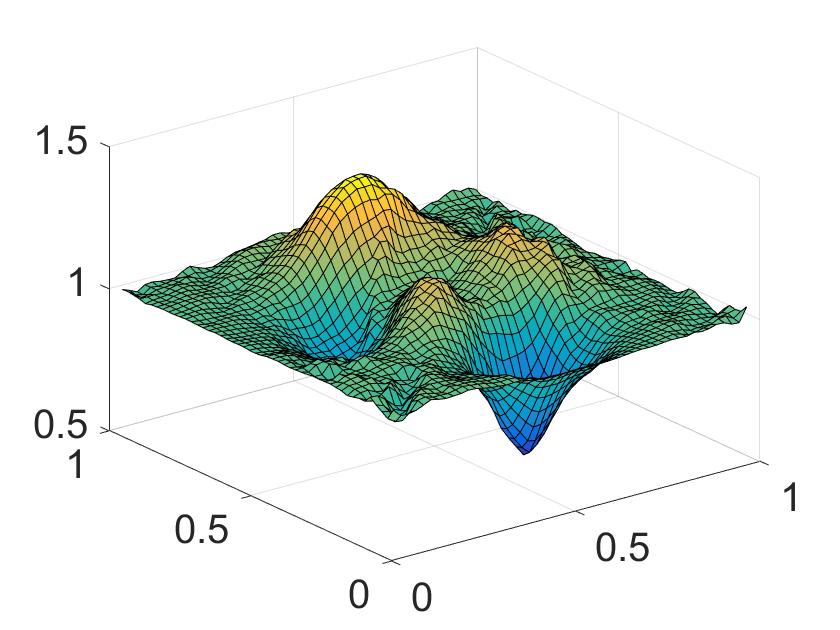} \\
(a) $\va=0.1$ &
(b) $\va=0.06$ &
(c) $\va=0.01$
\end{tabular}
\caption{Example 3. The reconstruction of $\gamma$ from noised measurements when $k_0=2$, $\delta=10\%$ for different $\va$ with relative errors RError=$2.46\times 10^{-2}$, $2.04\times 10^{-2}$ and $1.74\times 10^{-2}$ in (a,b,c), respectively.}
\label{Exp3.2}
\end{figure}

\begin{figure}[htb]
\centering
\includegraphics[scale=0.4]{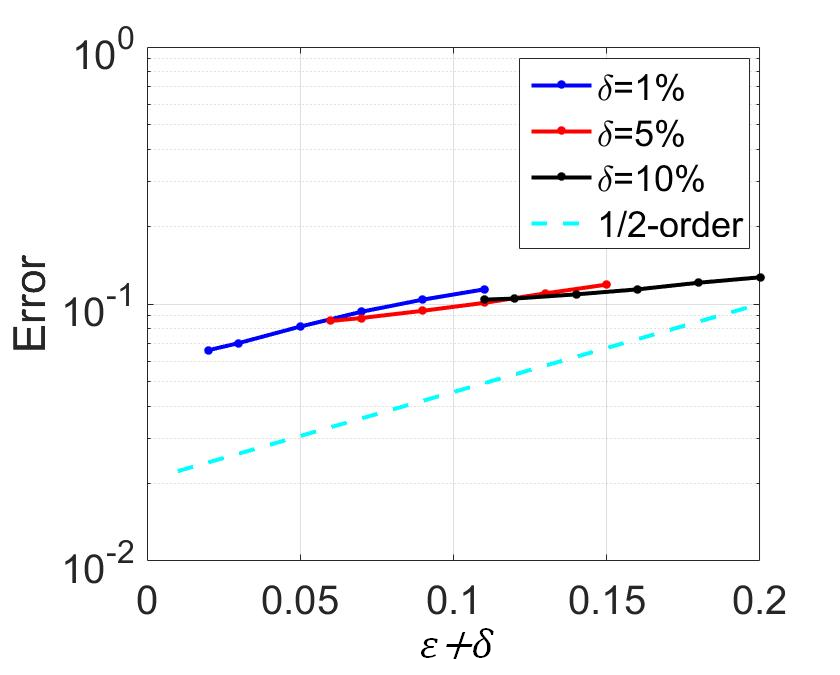}
\caption{Example 3. Numerical errors of the reconstruction for different noise level $\delta$ and $\va$.}
\label{Exp3.3}
\end{figure}

{\bf Example 4.} Finally, we consider the reconstruction of a piecewise constant conductivity, see Figure~\ref{Exp4.1}. For simplicity, the exact $u_\sigma$ in $\Omega$ is set to be
\bea
u_\sigma=\cos(x_1-0.5)e^{x_2}.
\eea
The reconstruction from perturbated point measurement with $10\%$ random noise presented in Figure~\ref{Exp4.2} shows the efficiency of the proposed method.

\begin{figure}[htb]
\centering
\includegraphics[scale=0.15]{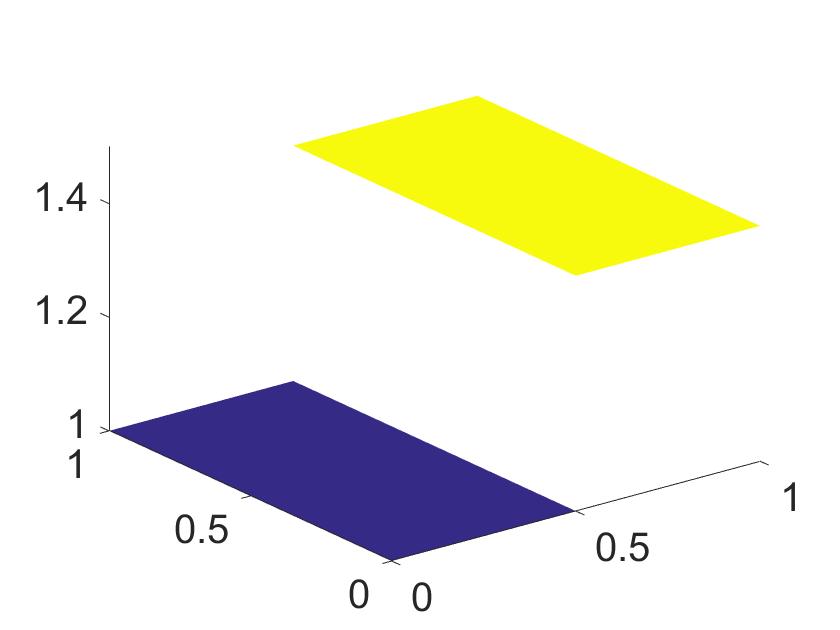}
\caption{Example 4. The exact $\gamma$.}
\label{Exp4.1}
\end{figure}

\begin{figure}[htb]
\centering
\begin{tabular}{ccc}
\includegraphics[scale=0.12]{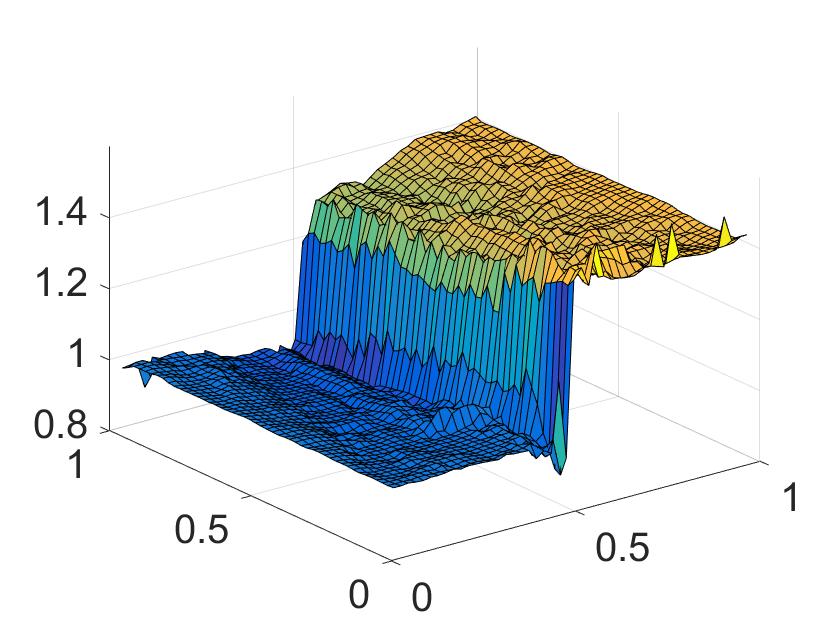} &
\includegraphics[scale=0.12]{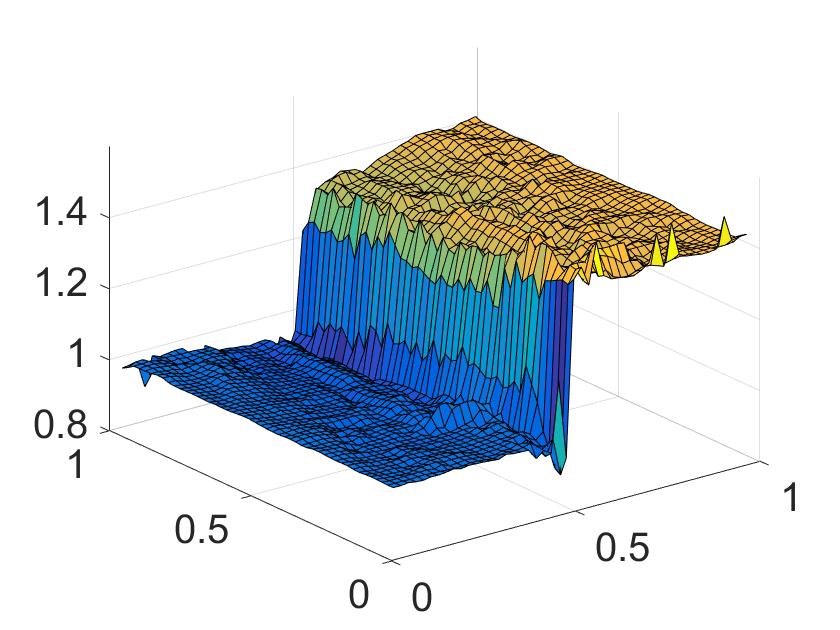} &
\includegraphics[scale=0.12]{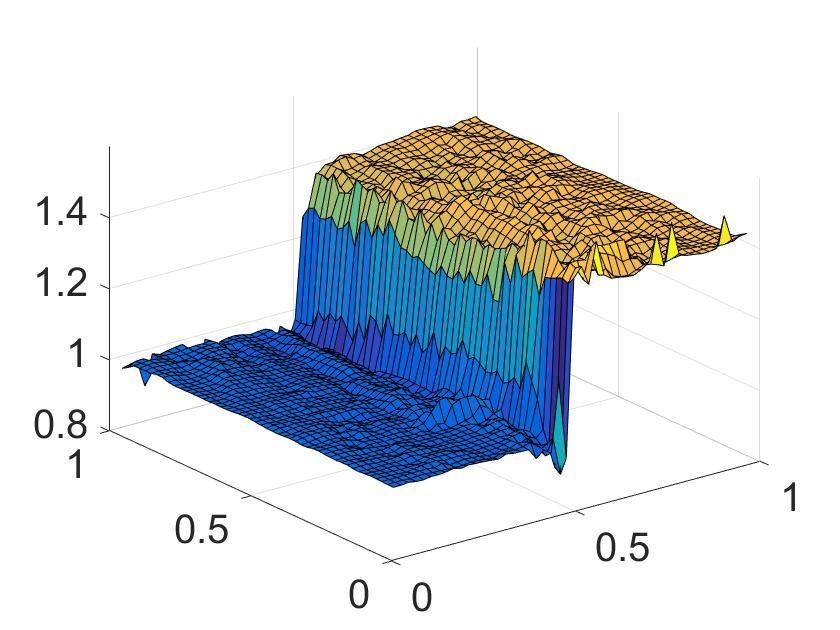} \\
(a) $\va=0.1$ &
(b) $\va=0.06$ &
(c) $\va=0.01$
\end{tabular}
\caption{Example 4. The reconstruction of $\gamma$ from noised measurements when $k_0=2$, $\delta=10\%$ for different $\va$ with relative errors RError=$3.79\times 10^{-2}$, $2.88\times 10^{-2}$ and $2.29\times 10^{-2}$ in (a,b,c), respectively.}
\label{Exp4.2}
\end{figure}

\appendix
\section{}
\label{App}

\begin{lemma}\label{tracelemma}({\bf Traces and integration by parts}  \cite{EGC, PE12} )\\
The trace operator
\bea
\gamma: C^0(\overline{\Omega})\ni v\mapsto \gamma(v):=v|_{\partial\Omega}\in L^2(|\beta\cdot \nu|,\partial\Omega)
\eea
extends continuously to $V$, meaning that there is $C_\gamma$ such that, for all $v\in V$,
\bea
\|\gamma(v)\|_{L^2(|\beta\cdot n|,\partial\Omega)}\le C_\gamma\|v\|_{V}.
\eea
Moreover, the following integration by parts formula holds true: For all $v,w\in V$,
\bea
\int_\Omega\left[(\beta\cdot\nabla v)w+(\beta\cdot\nabla w)v +(\nabla\cdot\beta)vw\right]dx=\int_{\partial\Omega}(\beta\cdot \nu)\gamma(v)\gamma(w)ds.
\eea
\end{lemma}

\begin{theorem}({\bf The Banach-Ne\v{c}as-Babu\v{s}ka Theorem} \cite{R08})\\\label{BNB}
Let $X$ be a Banach space and let $Y$ be a reflexive Banach space. Let $a\in\mathcal{L}(X\times Y,\R)$ and let $f\in Y'$. Then the problem:
\bea
\mbox{Find}\;\; u\in X\;\;\mbox{such that}\;\;a(u,w)=\langle f,w\rangle_{Y',Y}\;\;\mbox{for all}\;\; w\in Y.
\eea
is well-posed if and only if:\\
(i). There is $C_{sta} > 0$ such that
\bean
\label{cond1}
C_{sta}\|v\|_X\le \sup_{0\ne w\in Y}\frac{a(v,w)}{\|w\|_Y},\quad \forall v\in X.
\eean
(ii). For all $w\in X$,
\bean
\label{cond2}
(\forall v\in X, a(v,w)=0)\Longrightarrow (v=0).
\eean
\end{theorem}

\begin{proof}[Lemma~\ref{Lemmaunique}] Further $c>0$ is a constant that only depends on $g$,  $\sigma$, $\Omega$, and eventually on $\eta$. Since $u_\sigma$ is  a solution to the system~\eqref{cond},  $|\nabla u_\sigma|$ is  a Mukenhoupt weight.
Indeed there exists a contant $p>1 $ depending only 
on $g$,  $\sigma$ and $\Omega$ such that  the following inequality (Theorem 1.1 in  \cite{GL1})
\bean \label{u1}
\left(\frac{1}{|B_r(x)|} \int_{B_r(x)} |\nabla u_\sigma| dy\right)
\left(\frac{1}{|B_r(x)|} \int_{B_r(x)} |\nabla u_\sigma|^{-\frac{1}{p-1}} dy\right)^{p-1} \leq c,
\eean
holds for all $r\in (0, \eta)$, and $x\in \Omega_\eta$. 

The following behavior is related to the unique continuation properties of  solutions
to elliptic equations in a divergence form (Corollary 3.1 in \cite{CT}).
\bean \label{u2}
 cr^\beta \leq \frac{1}{|B_r(x)|} \int_{B_r}|\nabla u_\sigma| dy,
\eean
where $\beta \geq 0$ is a constant that only depends on $g$,  $\sigma$ and $\Omega$. The constants $\beta$ and $p>1$ are  related to the vanishing order of $\nabla u_\sigma$ in $\Omega_\eta$. Combining inequalities \eqref{u1} and \eqref{u2}, we obtain
\bean \label{u3}
 \int_{B_r(x)} |\nabla u_\sigma|^{-\frac{1}{p-1}} dy \leq cr^{\frac{-\beta+n}{p-1}},
\eean
for all $r\in (0, \eta)$, and $x\in \Omega_\eta$. 

Fix now $r= \frac{\eta}{2}$. There exist $N\in \mathbb N$ and $x_j, \,  1\leq j \leq N,$  that only depend on $\eta$ and
$\Omega$ such that  $ \Omega_\eta \subset  \cup_{j=1}^N B_{r}(x_j)$. Let $(\phi_j)_{1\leq j \leq N}$ be a partition of unity subordinate to the covering $\cup_{j=1}^N B_{r}(x_j)$. We deduce from  \eqref{u3}, the following estimates
\bean \label{u4}
\int_{\Omega_\eta}  |\nabla u_\sigma|^{-\frac{1}{p-1}} dy  &=& \sum_{j=1}^N
\int_{\Omega_\eta}  |\nabla u_\sigma|^{-\frac{1}{p-1}} \phi_j dy \nonumber\\
&\le& \sum_{j=1}^N \|\phi_j\|_{L^\infty(\Omega_\eta)} \int_{B_{r}(x_j)}  |\nabla u_\sigma|^{-\frac{1}{p-1}} dy \nonumber\\
&\le&  c.
\eean

Recall $\Omega_\eta\setminus \overline{\Omega^t_\eta} = \left\{x\in \Omega_\eta:  \,  |\nabla u_\sigma| <t \right\}$. Assuming that $\Omega_\eta\setminus \overline{\Omega^t_\eta}$ is not empty, we infer from \eqref{u4}   the following inequality
\bea
t^{-\frac{1}{p-1}} |\Omega_\eta\setminus \overline{\Omega^t_\eta}|
\leq  \int_{\Omega_\eta\setminus \overline{\Omega^t_\eta}}  |\nabla u_\sigma|^{-\frac{1}{p-1}} dy \leq \int_{\Omega_\eta
} |\nabla u_\sigma|^{-\frac{1}{p-1}} dy \leq c,
\eea
which in turn leads to
\bea
0 \leq |\Omega_\eta\setminus \overline{\Omega^t_\eta}| \leq c t^{\alpha},
\eea
with $\alpha:= \frac{1}{p-1}>0$.
\end{proof}

\bibliographystyle{siamplain}
\bibliography{references}
\end{document}